\documentclass[11pt,a4paper]{amsart}
\usepackage{amsmath,amsthm,amsfonts,graphicx,amssymb,amscd}
\usepackage{graphicx}
\usepackage{hyperref}
\newtheorem{thm}{Theorem}[section]

\newtheorem*{thmA8}{Theorem A.8}

\newtheorem{cor}[thm]{Corollary}
\newtheorem{lem}[thm]{Lemma}
\newtheorem{prop}[thm]{Proposition}
\theoremstyle{definition}
\newtheorem{defn}[thm]{Definition}

\usepackage[top=1.5in, bottom=1in, left=1.13in, right=1.13in]{geometry}

\begin{document}

\title{Conformal limits of grafting and Teichm\"{u}ller rays and their asymptoticity}
\author{Subhojoy Gupta}
\address{Center for Quantum Geometry of Moduli Spaces, Ny Munkegade 118, DK 8000 Aarhus C, Denmark.}
\email{sgupta@qgm.au.dk}
\date{\today}

\begin{abstract} We show  that any grafting ray in Teichm\"{u}ller space is (strongly) asymptotic to some Teichm\"{u}ller geodesic ray.  As an intermediate step we introduce surfaces that arise as limits of these degenerating Riemann surfaces.  Given a grafting ray, the proof involves finding a Teichm\"{u}ller ray with a conformally equivalent limit, and building  quasiconformal maps of low dilatation between the surfaces along the rays. Our preceding work had proved the result for rays determined by an arational lamination or a multicurve, and the unified approach here gives an alternative proof of the former case.
\end{abstract}

\maketitle

\section{Introduction}

Grafting rays and Teichm\"{u}ller rays are one-parameter families of (marked) conformal structures on a surface of genus $g$, that is, they are paths in Teichm\"{u}ller space $\mathcal{T}_g$, that arise in two disparate ways.  On one hand, Teichm\"{u}ller rays are geodesics in the Teichm\"{u}ller metric on $\mathcal{T}_g$, that relate to the complex-analytic structures on Riemann surfaces. On the other, grafting is an operation that has to do with the uniformizing hyperbolic structures (we shall assume throughout that $g\geq 2$), or more generally, complex projective structures on a surface (see \cite{KamTan}, \cite{Tan}). Both rays are determined by the data of a Riemann surface $X$ which is the initial point and a measured lamination $\lambda$ that determines a direction. In our preceding paper (\cite{Gup1}) we proved a strong comparison between these two for a ``generic" lamination in the space of measured laminations $\mathcal{ML}$, and here we extend that to the most general result.\\

Recall that two rays $\Theta$ and $\Psi$ in $\mathcal{T}_g$ are said to be asymptotic if the \textit{Teichm\"{u}ller distance} (defined in \S2) between them goes to zero, after reparametrizing if necessary. Here we establish the following result:

\begin{thm}[Asymptoticity]\label{thm:thm1} Let $(X,\lambda)\in \mathcal{T}_g\times \mathcal{ML}$. Then there exists a $Y\in \mathcal{T}_g$ such that the grafting ray determined by $(X,\lambda)$ is asymptotic to the Teichm\"{u}ller ray determined by $(Y,\lambda)$.
\end{thm}

In the preceding paper we had proved the above statement assuming $\lambda$ was \textit{arational} or was a multi-curve. This assumption was mild, since this includes a full-measure subset of $\mathcal{ML}$, and was enough for proving certain density results in moduli space (Corollary 1.2 and Theorem 1.5 of \cite{Gup1}). In this paper, we remove any such assumption. This is the first step of a fuller comparison,  in forthcoming work,  between the dynamics of grafting and the much-studied Teichm\"{u}ller geodesic flow (\cite{MasErg}, \cite{Veech2} - see \cite{MasSurv} for a survey). For some earlier work in this direction see \cite{DiazKim} and  the ``fellow-travelling" result in \cite{ChoiDumRaf}.  \\

The strategy of the proof of Theorem \ref{thm:thm1} generalizes that for the multicurve case in \cite{Gup1}, and yields an alternative proof of the arational case. A key intermediate step is to produce conformal limits of grafting rays and Teichm\"{u}ller rays. See \S5 for definitions and details, and a more precise version of the following:
\newpage
\begin{prop}[Conformal limits] For $\{Z_t\}_{t\geq 0} \subset \mathcal{T}_g$ a grafting or Teichm\"{u}ller ray, there exists a noded Riemann surface $Z$ such that  for any $\epsilon>0$, for all sufficiently large $t$ there are $(1+\epsilon)$-quasiconformal embeddings $f_{t,\epsilon}:Z_t\to Z$ whose images are isotopic in $Z$, and  exhaust $Z$ as $t\to \infty$.
\end{prop}

In particular, the singular flat surfaces along a Teichm\"{u}ller ray limit to  a ``generalized {half-plane surface}" which has an infinite-area metric structure comprising half-planes and half-infinite cylinders (see \S A.1 for a definition).  Equivalently, these are Riemann surfaces equipped with certain meromorphic quadratic differentials of higher order poles (see \cite{Gup25}). As discussed in the Appendix, the following result is an easy generalization of  the main result of  \cite{Gup25}:

\begin{thmA8} For any Riemann surface $Z$ with a set of marked points $P$, there exists a meromorphic quadratic differential with prescribed local data (orders, residues and leading order terms) at $P$ that induces a generalized half-plane structure.
\end{thmA8}

The proof of Theorem \ref{thm:thm1} proceeds by equipping the limit of the grafting ray with an appropriate ``generalized half-plane differential" using the above result, and reconstructing a Teichm\"{u}ller ray that limits to that half-plane surface. The asymptoticity is established by constructing quasiconformal maps of low dilatation from (large) grafted surfaces to this ray. This requires building an appropriate decomposition of the surfaces, and adjusting the conformal map between the limits on each piece. This uses some of our previous work in \cite{Gup1} together with some  technical analytical lemmas.\\




\textbf{Outline of paper.} In \S3 we recall some of the results from \cite{Gup1}  which we employ in this paper, and in \S4 we provide a compilation of the technical lemmas involving quasiconformal maps. In  \S5 we define conformal limits of  Riemann surfaces, and construct limits of grafting and Teichm\"{u}ller rays. Following an outline of the proof in \S6, we complete the proof of Theorem \ref{thm:thm1} in \S7. The Appendix recalls some of the work in \cite{Gup25} and outlines the generalization to Theorem \ref{thm:hpd} that is used in \S7.\\

\textbf{Acknowledgements.}
Most of this work  was done when I was a graduate student at Yale, and I wish to thank my advisor Yair Minsky for his generous help and patient guidance during this project. I also thank Ursula Hamenst\"{a}dt for  helpful discussions, and acknowledge the support of the Danish National Research Foundation grant DNRF95 and the Centre for Quantum Geometry of Moduli Spaces (QGM) where this work was completed.

\section{Preliminaries}

\subsection{Teichm\"{u}ller space}

The Teichm\"{u}ller space $\mathcal{T}_{g}$ is the space of marked Riemann surfaces of genus $g$. More precisely, if $S$ denotes a (fixed) topological surface of genus $g$, it is the collection of pairs:

\begin{center}
$\mathcal{T}_g = \{(f, \Sigma) \vert$ $ f:S_{g}\to \Sigma$ is a homeomorphism,\\ $\Sigma$ is a Riemann surface $\}$/$\sim$
\end{center}
where the equivalence relation is:
\begin{center}
$(f, \Sigma) \sim  (g, \Sigma^\prime)$
\end{center}
if there is a conformal homeomorphism $h:\Sigma \to \Sigma^\prime$ such that $h\circ f$ is isotopic to $g$. This definition extends to punctured surfaces.\\

\textbf{Teichm\"{u}ller metric.} The \textit{Teichm\"{u}ller distance} between two points $X$ and $Y$ in $\mathcal{T}_{g}$ is:
\begin{equation*}
d_{\mathcal{T}}(X,Y) = \frac{1}{2}\inf\limits_{f} \ln K_f
\end{equation*}
where $f:X\to Y$ is a quasiconformal homeomorphism preserving the marking and $K_f$ is its quasiconformal dilatation. See \cite{Ahl} for definitions.

\subsection{Quadratic differential metric}

A \textit{quadratic differential} $q$ on $X\in \mathcal{T}_g$ is  a differential of type $(2,0)$ locally of the form $q(z)dz^2$. It is said to be \textit{holomorphic} (or \textit{meromorphic}) when $q(z)$ is holomorphic (or meromorphic).\\

A meromorphic quadratic differential $q\in Q(X)$ defines a conformal metric (also called the $q$-metric) given in local coordinates by $\lvert q(z)\rvert \lvert dz\rvert^2$ which is flat away from the zeroes and poles. At the zeroes, the metric has a cone singularity of angle $(n+2)\pi$ (here $n$ is the order of the zero). At the poles, a  neighborhood  either has a ``fold" (for a simple pole), or is isometric to a half-infinite cylinder (for a pole of order two) or has the structure of a \textit{planar end} defined below (for higher-order poles). See \cite{Streb}, and \cite{Gup25} for a discussion. \\

In what follows we think of a euclidean half-plane as the region on one side of the vertical ($y$-) axis on the plane, which we identify with $\mathbb{R}$.

\begin{defn}[Planar-end, metric residue]\label{defn:pend}
Let $\{H_i\}$ for $1\leq i\leq n$ be a cyclically ordered collection of half-planes with rectangular ``notches" obtained by deleting, from each, a rectangle of horizontal and vertical sides adjoining the boundary, with the boundary segment having end-points $a_i$ and $b_i$, where $a_i<b_i$. A \textit{planar end} is obtained by gluing the interval $[b_i,\infty)$ on $\partial H_i$ with $(-\infty, a_{i+1}]$ on $H_{i+1}$ by an orientation-reversing isometry. Such a surface is homeomorphic to a punctured disk, and has a \textit{metric residue} defined to be absolute value of the alternating sum $\sum\limits_{i=1}^n (-1)^{i+1}(b_i - a_i)$.
\end{defn}

\begin{figure}[h]
  \centering
  \includegraphics[scale=0.5]{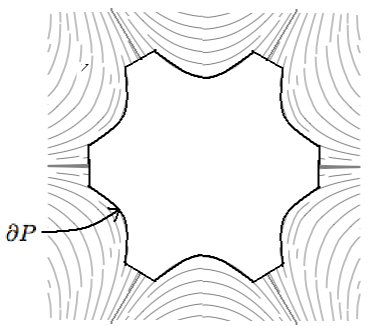}\\
  \caption{A planar end $P$ involving six half-planes, shown with its vertical foliation, and polygonal boundary.}
\end{figure}

\textit{Remark.} A planar end of order $n$ and metric residue $a$ is isometric to a planar domain containing $0$, equipped with the meromorphic quadratic differential $q_0 = \left(\frac{1}{z^{n+2}} + \frac{ia}{z^{n/2 +2}}\right)dz^2$. (See \S3.3 of \cite{Gup25}.)\\

\textbf{Measured foliations.}
A holomorphic quadratic differential $q\in Q(X)$ also determines a \textit{horizontal foliation} on $X$ which we denote by $\mathcal{F}_h(q)$, obtained by integrating the line field of vectors $\pm v$ where the quadratic differential is real and positive, that is $q(v,v)\geq 0$. Similarly, there is a \textit{vertical foliation} $\mathcal{F}_v(q)$ consisting of integral curves of directions where $q$ is real and negative.\\

Note that we can also talk of horizontal or vertical \textit{segments} on the surface, as well as horizontal and vertical \textit{lengths}.\\

The foliations above are \textit{measured}: the measure of an arc transverse to $\mathcal{F}_h$ is given by its \textit{vertical} length, and the transverse measure for $\mathcal{F}_v$ is given by horizontal lengths. Such a measure is invariant by isotopy of the arc if it remains transverse with endpoints on leaves.\\

For a fixed $X$, a quadratic differential $q\in Q(X)$ is determined uniquely by its horizontal (or vertical) foliation (\cite{HubbMas}).

\subsection{Geodesic laminations}
A \textit{geodesic lamination} on a hyperbolic surface is a closed subset of the surface which is a 
union of disjoint simple geodesics. Any geodesic lamination $\lambda$ is a  disjoint union of sublaminations
\begin{equation}\label{eq:lamdec}
\lambda = \lambda_1\cup \lambda_2 \cup\cdots \lambda_N \cup \gamma_1 \cup \gamma_2\cdots \gamma_M
\end{equation}
where $\lambda_i$s are minimal components (with each half-leaf dense in the component) which consist of uncountably many geodesics (a Cantor set cross-section) and the $\gamma_j$s are isolated geodesics (see \cite{CassBl}).\newline

A \textit{measured} geodesic lamination is equipped with a transverse measure $\mu$, that is a measure on arcs transverse to the lamination which is invariant under sliding along the leaves of the lamination. It can be shown that for the support of a measured lamination the isolated leaves in (\ref{eq:lamdec}) above are weighted simple closed curves (ruling out the possibility of isolated geodesics spiralling onto a closed component).  We call a lamination \textit{arational} if it consists of a single minimal component that is \textit{maximal}, that is, whose complementary regions are ideal hyperbolic triangles.\\

Any measured geodesic lamination corresponds to a unique {measured foliation} of the surface, obtained by `collapsing' the complementary components. Conversely, any measured foliation can be `tightened' to a geodesic lamination, and hence the two are equivalent notions (see, for example, \cite{Lev} or \cite{Kap}).

\subsection{Train tracks}
A \textit{train-track} on a surface is an embedded $C^1$ graph with a labelling of incoming and outgoing half-edges at every vertex (\textit{switch}). A \textit{weighted} train-track comes with an assignment of non-negative real numbers (\textit{weights}) to the edges (\textit{branches}) such that at every switch, the sums of the weights of the incoming and outgoing branches are equal. A standard reference is \cite{PenHar}. The leaves of a measured lamination (or a sufficiently long simple closed curve) lie close to such a train-track, and the transverse measures provide the weights. This provides a convenient combinatorial encoding of a lamination (see, for example, \cite{FLP} or \cite{Thu0}). 

\subsection{Teichm\"{u}ller rays}

A geodesic ray in the Teichm\"{u}ller metric (defined in \S2.2) starting from a point $X$ in $\mathcal{T}_g$ and in a direction determined by a holomorphic quadratic differential $\phi\in Q(X)$ (or equivalently by a lamination $\lambda\in \mathcal{ML}$) is obtained by starting with the singular flat surface $X$ and scaling the metric along the horizontal foliation $\mathcal{F}_h(\phi)$ by a factor of $e^{2t}$.\\

Note that we shall use the convention that the lamination $\lambda$ is the \textit{vertical} foliation, so the stretching above is transverse to it.\\

\textit{Remark.} By a conformal rescaling, this is equivalent to the usual definition involving scaling the horizontal direction by a factor $e^{t}$ and the vertical direction by a factor $e^{-t}$. This has the advantage of preserving the $q$-area which is superfluous in our case as we shall be looking for certain geometric limits of infinite area (\S5.3.3).

\subsection{Grafting} 
The \textit{conformal} grafting map 
\begin{center}
$gr:\mathcal{T}_g\times \mathcal{ML}\to \mathcal{T}_g$
\end{center}
is defined as follows: for a simple closed curve $\gamma$ with weight $s$, grafting a hyperbolic surface $X$ along $s\gamma$ inserts a euclidean annulus of width $s$ along the geodesic representative of $\gamma$ (see Figure 2). The resulting surface acquires a $C^{1,1}$-metric (see \cite{ScanWolf} and also \cite{KulPink}) and a new conformal structure which is the image $gr(X,s\gamma)$. Grafting for a \textit{general} measured lamination $\lambda$ is defined by taking the limit of a sequence of approximations of $\lambda$ by weighted simple closed curves, that is, a sequence $s_i\gamma_i \to \lambda$ in $\mathcal{ML}$.\\

\textit{Notation.}  We often write $gr(X,\lambda)$ as $gr_\lambda(X)$.\\

The hybrid euclidean-and-hyperbolic metric on the grafted surface is called the \textit{Thurston metric}. The length in the Thurston metric on $gr_{\lambda}X$ of an arc $\tau$  intersecting $\lambda$, is its hyperbolic length on $X$ plus its transverse measure. We shall refer to the \textit{euclidean} length  (of any arc) being this total length with the hyperbolic length subtracted off. \\

For further details, as well as the connection with complex projective structures, see \cite{Dum0} for an excellent survey.\\

It is known that for any fixed lamination $\lambda$, the grafting map ($X\mapsto gr_{\lambda}X$) is a self-homeomorphism of $\mathcal{T}_g$ (\cite{ScanWolf}).

\begin{defn}\label{defn:gray}
A \textit{grafting ray} from a point $X$ in $\mathcal{T}_g$ in the `direction' determined by a lamination $\lambda$ is the $1$-parameter family $\{X_t\}_{t\geq 0}$ where $X_t = gr_{t\lambda}(X)$.
\end{defn}

\begin{figure}
  \centering
  \includegraphics[scale=0.4]{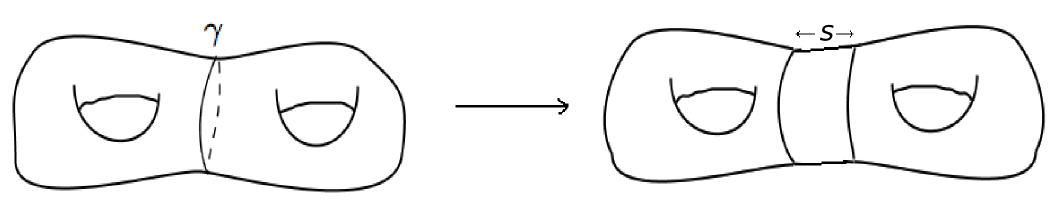}\\
  \caption{The grafting map along the weighted curve $s\gamma$. }
\end{figure}

\section{Geometry of grafted surfaces}

The surface obtained by grafting a hyperbolic surface $X$ along a measured lamination $\lambda$ acquires a conformal metric (the \textit{Thurston metric}) that is euclidean in the grafted region and hyperbolic elsewhere - see the preceding description in \S2.6.  The case of $\lambda$ a weighted simple closed curve is easily described: the metric comprises a euclidean cylinder of length equal to the weight, inserted at the geodesic representative of $\lambda$ (see Figure 2).  A general measured lamination has more complicated structure, with uncountably many geodesics winding around the surface. In this section we briefly recall some of the previous work in \cite{Gup1} that develops an understanding of the underlying conformal structure of the grafted surface when one grafts along such a lamination.\\

Though the finer structure of a general measured lamination is complicated, it can be combinatorially encoded as a train-track  on the surface (see \S2.4) which also allows a convenient description of the grafted metric. For sufficiently small $\epsilon>0$, an $\epsilon$-thin Hausdorff neighborhood (in other words, a slight thickening) $\mathcal{T}_\epsilon$ of a measured lamination $\lambda$ is such an embedded graph -  the leaves of the $\lambda$ run along embedded rectangles which are the branches, and the branch weights are the transverse measures of arcs across them.\\

The intuition is that as one grafts, the subsurface $\mathcal{T}_\epsilon$ widens in the transverse direction (along the ``ties" of the train-track), and conformally approaches a union of wide euclidean rectangles. The complement $X\setminus \lambda$ is unaffected by grafting: it may consist of ideal polygons or subsurfaces with moduli. This complementary hyperbolic part becomes negligible compared to the euclidean part of the Thurston metric, for a sufficiently large grafted surface. The results of \cite{Gup1}, which we recall briefly in this section, make this intuitive picture concrete.

\subsection{Train-track decomposition}

We summarize the construction (see \S4.1 of \cite{Gup1} for details)  of the afore-mentioned subsurface $\mathcal{T}_{\epsilon}\subset X$ containing the lamination $\lambda$:\\

Recall that $\lambda$ has a decomposition  (\ref{eq:lamdec}) into minimal components and simple closed curves .\\

For each minimal component of $\lambda$, choose a transverse arc $\tau$ of hyperbolic length sufficiently small, depending on $\epsilon$ (see Lemma \ref{lem:thinrect}) and use the first return map of $\tau$  to itself (following the leaves of $\lambda$) to form a collection of rectangles  $R_1, R_2, \ldots R_N$ with vertical geodesic sides, and horizontal sides lying on $\tau$.\\

For each simple closed component, consider an annulus  of hyperbolic width $\epsilon$ containing it.\\

 We define 
\begin{equation}\label{eq:tedec}
\mathcal{T}_\epsilon = R_1\cup R_2\cup \cdots R_N \cup  A_1\cup \cdots A_M
\end{equation}
to be the union of these rectangles and annuli  on the surface $X$.  This contains (or \textit{carries}) the lamination $\lambda$.

\begin{figure}[h]
  \centering
  \includegraphics[scale=0.5]{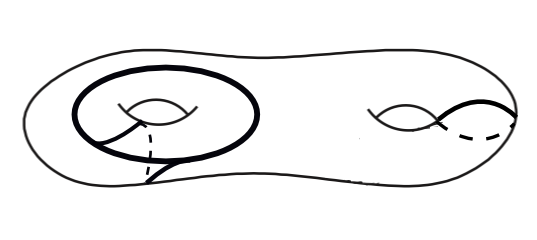}\\
  \caption{The train track $\mathcal{T}_\epsilon$ is a union of rectangular and annular branches - see (\ref{eq:tedec}). These contain the lamination, and widen as one grafts.}
\end{figure}

\begin{lem}[Lemma 4.2 in \cite{Gup1}] \label{lem:thinrect}
If the hyperbolic length of $\tau$ is sufficiently small, the height of each of the rectangles $R_1,R_2,\ldots R_N$ is greater than $\frac{1}{\epsilon^2}$, and the hyperbolic width is less than $\epsilon$.
\end{lem}

The complement of $\mathcal{T}_\epsilon$ is a union of subsurfaces $T_1,T_2,\ldots T_m$. 
We also describe a ``thickening" of $\mathcal{T}_\epsilon$ (or ``trimming" of the above subsurfaces) to ensure that $\lambda$ is contained \textit{properly} in $\mathcal{T}_\epsilon$:\\
For each geodesic side in $T_1,T_2,...T_m$, choose an adjacent thin strip (inside the region in the complement of $\mathcal{T}_\epsilon$)  and bounded by another geodesic segment ``parallel" to the sides,  and append it to the adjacent rectangle of the train-track.\\

 We continue to denote the collection of this slightly thickened rectangles by $R_1,R_2,\ldots R_N$, and their union by $\mathcal{T}_\epsilon$.\\

Along the grafting ray determined by $(X,\lambda)$, the rectangles get wider, and one gets a decomposition of each grafted surface $gr_{t\lambda}X$ into rectangles and annuli and these complementary regions, with the same combinatorics of the gluing.\\

As in \cite{Gup1}, the \textit{total width} of a rectangle in the above decomposition is the maximum width in the Thurston metric, and its \textit{euclidean width} is the total width minus the initial hyperbolic width.

\begin{lem}[See \S4.1.3 of \cite{Gup1}]\label{lem:widerect}
If $P$ is a rectangle or annulus as above on $X$, its euclidean width on $gr_{t\lambda}X$ is $t\mu(P)$ where $\mu(P)$ is the measure of an arc across $P$, transverse to $\lambda$. 
\end{lem}

\subsection{Almost-conformal maps}
Recall that by the definition of the Teichm\"{u}ller metric in \S2.1, if there exists a map $f:X\to Y$ which is $(1+\epsilon)$-quasiconformal, then 
\begin{equation*}
d_{\mathcal{T}}(X,Y) = O(\epsilon).
\end{equation*}
\textit{Notation.} Here, and throughout this article, $O(\alpha)$ refers to a quantity bounded above by $C\alpha$ where $C>0$ is some constant depending only on genus $g$ (which remains fixed), the exact value of which can be determined \textit{a posteriori}.\\

As in \cite{Gup1}, we shall use:

\begin{defn}\label{defn:ac}
An \textit{almost-conformal map} shall refer to a map which is $(1+ O(\epsilon))$-quasiconformal.
\end{defn}

We shall be constructing such almost-conformal maps between surfaces in our proof of the asymptoticity result  (Theorem \ref{thm:thm1}) and using the notions in this section.\\

In \S4.2 and \S6.1 of \cite{Gup1} we developed the notion of ``almost-isometries" and an ``almost-conformal extension lemma" that we restate (we refer to the previous paper for the proof).  

\begin{defn}\label{defn:ecisom}
A homeomorphism $f$ between two $C^1$-arcs on a conformal surface is an \textit{$(\epsilon,A)$-almost-isometry} if $f$ is continuously differentiable with dilatation $d$ (the supremum of the derivative of $f$ over the domain arc) that satisfies $\lvert d-1\rvert \leq \epsilon$ and such that the lengths of any subinterval and its image differ by an additive error of at most $A$.
\end{defn}

\textit{Remark.} We shall sometimes say `$(\epsilon, A)$-almost-isometric' to mean `$(M\epsilon,A)$-almost-isometric for some (universal) constant $M>0$'.

\begin{defn}\label{defn:ecgood}
A map $f$ between two rectangles is \textit{$(\epsilon,A)$-good} if it is isometric on the vertical sides and $(\epsilon,A)$-almost-isometric on the horizontal sides.
\end{defn}

\textit{Remark.} A ``rectangle" in the above definition refers to four arcs in any metric space intersecting at right angles, with a pair of opposite sides of equal length being identified as \textit{vertical} and the other pair also of identical length called \textit{horizontal}.\\

The following lemma from \cite{Gup1} (Lemma 6.8 in that paper) shall be used in the final construction in  \S7.5:

\begin{lem}[Qc extension]\label{lem:ext}
Let $R_1$ and $R_2$ be two euclidean rectangles with vertical sides of length $h$ and horizontal sides of lengths $l_1$ and $l_2$ respectively, such that  $l_1,l_2 >h$ and $\lvert l_1-l_2\rvert <A$, where $A/h \leq\epsilon$. Then any  map $f:\partial R_1\to \partial R_2$  that is $(\epsilon,A)$-good on the horizontal sides and isometric on the vertical sides has a $(1+C\epsilon)$-quasiconformal extension $F:R_1\to R_2$ for some (universal) constant $C>0$.
\end{lem}

\subsection{Model rectangles}

Recall that the train-track decomposition of $X$ persists as we graft along $\lambda$, with the {euclidean width} of the rectangles $R_1,R_2,\ldots,R_N$ and annuli $A_1,A_2\ldots A_M$  increasing along the $\lambda$-grafting ray. We denote the grafted rectangles and annuli on $gr_{2\pi t\lambda}X$ by $R_1^t,R_2^t,\ldots R_N^t, A_1^t, A_2^t,\ldots A_M^t$.\newline

The annuli are purely euclidean, but the rectangles have (typically infinitely many) hyperbolic strips through them. The euclidean width is proportional to the transverse measure (Lemma \ref{lem:widerect}). The \textit{total width} $w_i(t)$  is the maximum width of the rectangle $R_i^t$ in the Thurston metric on $gr_{2\pi t\lambda}X$, and we have
\begin{equation}\label{eq:wit}
\lvert w_i(t) - 2\pi tw_i\rvert <\epsilon
\end{equation}
since the initial \textit{hyperbolic} widths of the rectangles on $X$ is less than $\epsilon$ by Lemma \ref{lem:thinrect}.\newline

Though having a hybrid hyperbolic-and-euclidean metric, a rectangular piece $R^t$ from the collection $\{R_1^t,R_2^t,\ldots,R_N^t\}$ has an almost-conformal \textit{euclidean model rectangle}, as proved in \S 4.3 of \cite{Gup1} (see also Lemma 6.9 in that paper). We restate that result, and since it is crucial in our constructions, we provide a sketch of the proof:

\begin{lem}[Lemma 6.19 of \cite{Gup1}]\label{lem:model}
For any $t>1$, there is a $(1+ C\epsilon)$-quasiconformal map from $R^t$ to a euclidean rectangle of width $2\pi tw_i$ which is  $(\epsilon,\epsilon)$-good on the boundary. (Here $C>0$ is some universal constant.)
\end{lem}
\begin{proof}[Sketch of the proof]
One can approximate the measured lamination $\lambda \cap R$ passing through the rectangle (on the ungrafted surface) by \textit{finitely} many weighted arcs by a standard finite-approximation of the measure.   Grafting $R$ along this finite approximation $\lambda_i$ amounts to splicing in euclidean strips at the arcs, of width equal to their weights. The hyperbolic rectangle $R$  is $\epsilon$-thin (put $t=0$ in (\ref{eq:wit})).\\

\begin{figure}
  \centering
  \includegraphics[scale=0.6]{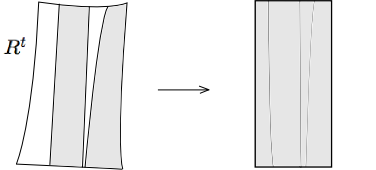}\\
  \caption{The thin hyperbolic pieces (shown unshaded) of the grafted rectangle can be mapped to the euclidean plane and together with a map of the euclidean pieces defines an almost conformal map to a rectangle (Lemma \ref{lem:model}).}
\end{figure}

 Working in the upper-half-plane model of the hyperbolic plane, one can map $R$ to the euclidean plane by an almost-conformal map that ``straightens" the horocyclic foliation across it. The images of the vertical sides are ``almost-vertical" and the finitely many euclidean strips that need to be spliced in can be mapped in by almost-conformal maps such that the images fit together to a form an "almost" euclidean rectangle (with almost-vertical sides). The condition that $t>1$ implies that this can finally be adjusted to a euclidean rectangle by a horizontal stretch with stretch-factors $\epsilon$-close to $1$ at each height. Call this composite almost-conformal map $f_i$. One then takes a limit of these $f_i$-s for a sequence of approximations  $\lambda_i \to \lambda\cap R$.
\end{proof}

\subsection{Asymptoticity in the arational case}

We briefly summarize the strategy of the proof of Theorem \ref{thm:thm1} in the case when the lamination $\lambda$ is arational. This was carried out in \cite{Gup1}, and we refer to that paper for details.\\

In this case the train-track $\mathcal{T}_\epsilon$ carrying $\lambda$ is maximal, that is, its complement is a collection of truncated ideal triangles, and one can build a \textit{horocyclic foliation} $\mathcal{F}$  transverse to the lamination. The branches widen along a grafting ray as more euclidean region is grafted in (Lemma \ref{lem:widerect}), and $\mathcal{F}$ lengthens. One can define a singular flat surface $\hat{X_t}$ obtained by collapsing the hyperbolic part of the Thurston metric on $X_t$ along $\mathcal{F}$, and it is not hard to check that these lie along a common Teichm\"{u}ller ray determined by a measured foliation equivalent to $\lambda$.\\

The maps of Lemma \ref{lem:model} from the rectangles on the grafted surface to euclidean rectangles piece together to form a quasiconformal map from $X_t$ to  $\hat{X_t}$ that is almost-conformal for \textit{most} of the grafted surface (Lemma 4.22 of \cite{Gup1}). For all sufficiently large $t$, this can be adjusted to an almost-conformal map of the entire surface by (a weaker version of) Lemma \ref{lem:qclem} in the next section. Since the singular flat surfaces $\hat{X_t}$ lie along a common Teichm\"{u}ller ray, as $\epsilon\to 0$ this proves Theorem 1.1 for arational laminations.

 \section{A quasiconformal toolkit}
In this section we collect some constructions and extensions of quasiconformal maps that shall be useful in the proof of Theorem \ref{thm:thm1}. This forms the technical core of this paper, and the reader is advised to skip it at first reading, and refer to the lemmas whenever they are used later.\\

Most of the results here are probably well-known to experts, however in our setting we need care to maintain almost-conformality of the maps (see \S3.2), and this aspect seems to be absent in the literature. For a glossary of known results we refer to the Appendix of \cite{Gup1}.\newline

Throughout, $\mathbb{D}$ shall denote the unit closed disk on the complex plane, and $B_r$ shall denote the closed disk of radius $r$ centered at $0$. Note that any quasiconformal map defined on the interior of a Jordan domain extends to a homeomorphism of the boundary.

\subsection{Interpolating maps}
We start with the following observation about the Ahlfors-Beurling extension that was used in a construction in \cite{Jones}:

\begin{lem}[Interpolating with identity]\label{lem:Jones}
Let $h:\partial \mathbb{D} \to \partial \mathbb{D}$ be a $(1 + \epsilon)$-quasisymmetric map. Then there exists an $0<s<1$ and a  homeomorphism $H:\mathbb{D}\to \mathbb{D}$ such that\newline
(1) $H$ is $(1+C\epsilon)$-quasiconformal.\newline
(2) $H\vert_{\partial \mathbb{D}} =h\vert_{\partial \mathbb{D}}$.\newline
(3) $H$ restricts to the identity map on $B_s$.\newline
Here $C>0$ is a universal constant, and $s$ above depends only on $\epsilon$.
\end{lem}
\begin{proof}
As in \S 2.4 of \cite{Jones}, lift $h$ to a homeomorphism $\tilde{h}:\mathbb{R}\to \mathbb{R}$ that satisfies $\tilde{h}(x+1) = \tilde{h}(x) +1$, and consider the Ahlfors-Beurling extension of $\tilde{h}$ to the upper half plane $\mathbb{H}$:
\begin{equation*}
F(x + iy) = \frac{1}{2}\int\limits_0^1 \tilde{h}(x+ty) +\tilde{h}(x-ty)dt + i\int\limits_0^1 \tilde{h}(x+ty) - \tilde{h}(x-ty) dt
\end{equation*}
It follows from the periodicity that 
\begin{equation*}
F(x+i) = x + i + c_0
\end{equation*}
where $c_0 = \int\limits_0^1\tilde{h}(t)dt - 1/2 \in [-1/2,1/2]$.\newline
We note that since we have used a locally conformal change of coordinates $w\mapsto e^{2\pi iw}$ between $\mathbb{H}$ and $\mathbb{D}$, we have that $\tilde{h}$ is also $(1 + O(\epsilon))$-quasisymmetric, and $F$ is almost-conformal.\newline
For $D>1$ we can define a map $F_1:\mathbb{H}\to \mathbb{H}$ which restricts to $F$ on $\mathbb{R} \times [0,1]$, and the identity map for $y\geq D$ and interpolates linearly on the strip $\mathbb{R} \times [1,D]$:
\begin{equation*}
F_1(x+iy) = x+iy + c_0\left(\frac{D-y}{D-1}\right)
\end{equation*}
For $D$ sufficiently large (greater than $1/C\epsilon$), $F_1$ is almost-conformal everywhere as can be checked by computing derivatives on the interpolating strip. Since $F_1(z + k) = F_1(z) + k$ for all $z\in \mathbb{H}$ and $k\in \mathbb{Z}$, it descends to an almost-conformal map $H:\mathbb{D}\to \mathbb{D}$ that restricts to $h$ on $\partial \mathbb{D}$ and the identity map on $B_s$ for $s = e^{-2\pi D}$.
\end{proof}

\textit{Notation.} In the statements of the following lemmas, we use the same $C$  to denote universal constants that might \textit{a priori} vary between the lemmas, since one can, if needed, take a maximum of them and fix a single constant that works for each.\\

The following corollary of the above lemma interpolates a quasiconformal map with the identity map on the  \textit{outer} boundary:

\begin{lem}\label{lem:ffix}
Let $f:\mathbb{D}\to \mathbb{D}$ be a $(1+\epsilon)$-quasiconformal map such that $f(0)=0$. Then there exists an $0<r_0<1$ and a map $F:\mathbb{D}\to \mathbb{D}$ such that\newline
(1) $F$ is $(1+C\epsilon)$-quasiconformal.\newline
(2) $F\vert_{B_{r_0}} = f\vert_{B_{r_0}}$.\newline
(3) $F\vert_{\partial \mathbb{D}}$ is the identity.\\
Here $C>0$ is a universal constant.
\end{lem}
\begin{proof}
Since $f$ is almost-conformal, so is $f^{-1}$, and the latter extends to the boundary and restricts to a homeomorphism $h:\partial \mathbb{D} \to \partial \mathbb{D}$ that is $(1 + O(\epsilon))$-quasisymmetric. By Lemma \ref{lem:Jones} there exists an almost-conformal extension $H:\mathbb{D}\to \mathbb{D}$ of $h$ that restricts to the identity on $B_{r_0}$ for sufficiently small $r_0$. The composition $H\circ f$ then is the required map $F$ that restricts to $f$ on $B_{r_0}$ and is identity on $\partial \mathbb{D}$.
\end{proof}

We shall generalize the previous lemma to obtain an interpolation of an almost-conformal map with a \textit{given} conformal map at the outer boundary. We first recall from \cite{Gup1} the following fact that we can fix such a conformal map to be the identity near $0$, without too much quasiconformal distortion.

\begin{lem}[Lemma 5.1 of \cite{Gup1}]\label{lem:gfix}
Let $g:\mathbb{D}\to g(\mathbb{D})\subset \mathbb{C}$ be a conformal map such that $g(0)=0$ and $g^\prime(0)=1$. Then there exists an $0<s<1$ and a map $G:\mathbb{D}\to g(\mathbb{D})$ such that\newline
(1) $G$ is $(1+\epsilon)$-conformal.\newline
(2) $G$ restricts to the identity map on $B_{s}$.\newline
(3) $G\vert_{\partial \mathbb{D}} = g\vert_{\partial \mathbb{D}}$.
\end{lem}

\textit{Remark.} By conjugating by the dilation $z\mapsto (1/r)z$ the above result holds (for some $0<s< r$)  if the conformal map $g$ is defined only on $B_r\subset \mathbb{D}$.

\begin{figure}
  \centering
  \includegraphics[scale=0.55]{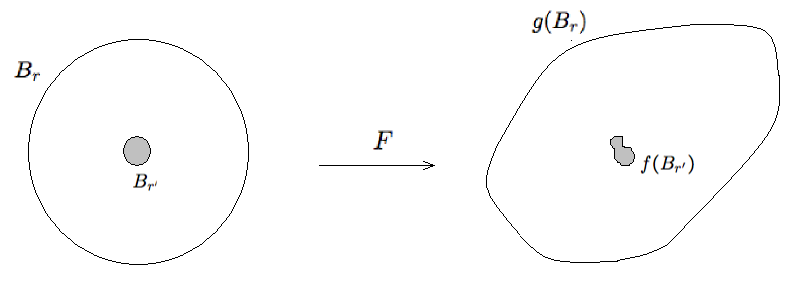}\\
  \caption{In Lemma \ref{lem:interp} we interpolate from the map $g$ on $\partial B_r$ to the almost-conformal map $f$ between the shaded regions.}
\end{figure}

\begin{lem}[Interpolation]\label{lem:interp}
Let $f:\mathbb{D}\to \mathbb{D}$ be a $(1+\epsilon)$-quasiconformal map such that $f(0)=0$, and let $g:B_r\to g(B_r)\subset \mathbb{D}$ , for some $0<r<1$, be a conformal map such that $g(0)=0$ and $g^\prime(0)=1$. Then there exists an $0<r^\prime< r< 1$ and a map $F:B_r\to g(B_r)$ such that\newline
(1) $F\vert_{B_{r^\prime}} = f\vert_{B_{r^\prime}}$.\newline
(2) $F\vert_{\partial B_r} = g\vert_{\partial B_r}$.
(3) $F$ is $(1+C\epsilon)$-quasiconformal, for some universal constant $C>0$.
\end{lem}
\begin{proof}
By Lemma \ref{lem:gfix} (see also the above remark)  there exists an $0<s<r<1$ and an almost-conformal map $G:B_r\to g(B_r)$ that restricts to $g$ on $\partial B_r$ and is identity on $B_{s}$. Now we can apply Lemma \ref{lem:ffix} to the rescaled map $f_s = (1/s)f\vert_{B_s}:\mathbb{D}\to\mathbb{D}$ to get a map $F_s:\mathbb{D}\to \mathbb{D}$ that restricts to $f_s$ on some  $B_{s^\prime}$ for $0<s^\prime <1$ and the identity map on the boundary $\partial \mathbb{D}$. Rescaling back, we get a map $F_1:B_s\to B_s$ that restricts to the identity map on $\partial B_s$ and to $f$ on $B_{r^\prime}$ where $r^\prime = s^\prime s < s$. Since $F_1$ and $G$ are both identity on $\partial B_s$, $F_1$ together with the restriction of $G$ on $B_r\setminus B_s$ defines the required interpolation $F:B_r\to g(B_r)$.
\end{proof}

\subsection*{Boundary correspondence for annuli}
The following lemma  has been proved in \cite{Gup25} (for a weaker version see Lemma A.1 of \cite{Gup1}).  \\

\begin{lem}\label{lem:qclem}
For any $\epsilon>0$ sufficiently small, and $0\leq r\leq \epsilon$, a map
\begin{center}
$f:\mathbb{D}\setminus B_r\to \mathbb{D}$
\end{center}
that\\
(1) preserves the boundary $\partial \mathbb{D}$ and is a homeomorphism onto its image,\\
(2) is $(1+\epsilon)$-quasiconformal on $\mathbb{D}\setminus B_r$ \\
extends to a $(1 + C\epsilon)$-quasisymmetric map on the boundary, where $C>0$ is a universal constant.
\end{lem}

\textit{Remark.} The boundary correspondence for disks (the case when $r=0$) is proved in \cite{AB}, and we employ their methods in the proof.
 
  \begin{cor}\label{cor:corqclem}
Let $\epsilon>0$ be sufficiently small, and $U_0,U$ and $U^\prime$ be topological disks such that $U_0 \subset U$ and the annulus $A = U\setminus U_0$ has modulus larger than $\frac{1}{2\pi}\ln\frac{1}{\epsilon}$. Then for any conformal embedding $g:A\to U^\prime$ there is a $(1+C^\prime\epsilon)$-quasiconformal map $f:U\to U^\prime$ such that $f$ and $g$ are identical on $\partial U$. (Here $C^\prime>0$ is a universal constant.)
\end{cor}
\begin{proof}
By uniformizing, one can assume that $U=U^\prime = \mathbb{D}$ and $U_0 \subset B_r$ where $r\leq \epsilon$ by the condition on modulus. Now we are in the setting of the previous lemma, since $g$ being conformal is also $(1 + \epsilon)$-quasiconformal on $A$. One can hence conclude that $g$ extends to a $(1+C\epsilon)$-quasisymmetric map of the boundary, which by the Ahlfors-Beurling extension (see \cite{AB}) extends to an $(1+C^\prime\epsilon)$-quasiconformal map of the entire disk, which is our required map $f$. 
\end{proof}

There is also a converse to Lemma \ref{lem:qclem}, which we prove using Lemma \ref{lem:Jones}:

\begin{lem}\label{lem:ABann}
Let $A$ be an annulus on the plane bounded by circles $\partial B_r$ and $\partial \mathbb{D}$ where $0<r<1$. Let $f_1:\partial \mathbb{D} \to \partial \mathbb{D}$ and $f_2:\partial B_r \to \partial B_r$ be $(1 + \epsilon)$-quasisymmetric maps. Then if $r$ is sufficiently small there is a $(1+C\epsilon)$-quasiconformal map $F:A\to A$ such that $F\vert_{\partial B_r} = f_2$ and $F\vert_{\partial \mathbb{D}} = f_1$. (Here $C>0$ is a universal constant.)
\end{lem}
\begin{proof}
Applying Lemma \ref{lem:Jones}, we can extend $f_1$ to an almost-conformal map $F_1:\mathbb{D}\to \mathbb{D}$ that is the identity map on $B_s$, for some $0<s<1$.  If  $r<s$ then by inverting across the circle $\partial B_r$ and scaling by a factor of $r$, we can apply the lemma again to construct a map $F_2^1: B_r^c\to B_r^c$ such that $F_2$ restricts to $f_2$ on $\partial B_r$ and is the identity map on $\partial \mathbb{D}$. By the usual Ahlfors-Beurling extension, there exists an extension $F_2^2:B_r\to B_r$ that is almost-conformal and restricts to $f_2$ on $\partial B_r$. The maps $F_2^1\vert_{\mathbb{D}\setminus B_r}$ and $F_2^2$ then define a map $F_2:\mathbb{D}\to \mathbb{D}$ that restricts to $f_2$ on $B_r$ and is identity on $\partial \mathbb{D}$. The required map $F:A\to A$ is then the restriction of the composition $F_2\circ F_1$ to the annular region $A$.
\end{proof}

\begin{cor}[Sewing annuli]\label{cor:glueAnn} Let $A$ be an annulus with core curve $\gamma$, such that $A\setminus \gamma = A_1\sqcup A_2$. Let $B_1, B_2$ be two other annuli, and let $B = B_1\sqcup_{\partial_0} B_2$ be the annulus obtained by gluing one boundary component of each by a $(1+C\epsilon)$-quasisymmetric map. Assume all $A_i,B_i$ (for $i=1,2$) have moduli greater than $\frac{1}{2\pi}\ln\frac{1}{\epsilon}$, and suppose $f_i:A_i\to B_i$ are $(1+\epsilon)$-quasiconformal maps. Then there exists a $(1+C^\prime\epsilon)$-quasiconformal map $F:A \to B$ that agrees with $f_1$ and $f_2$ on the boundary components of $A$. (Here $C, C^\prime>0$ are universal constants.)
\end{cor}
\begin{proof}
By uniformizing to a round annulus on the plane, we can assume each $A_i, B_i$ (for $i=1,2$) is the unit disk $\mathbb{D}$ with a subdisk of radius smaller than $\epsilon$ excised from it. By Lemma \ref{lem:qclem} we have that $f_i$ extend to a $(1 + C \epsilon)$-quasisymmetric map of the boundaries. These maps on the boundary might not agree with the quasisymmetric gluing maps of the annuli, but differ by a post-composition by a $(1+C\epsilon)$-quasisymmetric map of the circle (compositions of $K$-quasisymmetric maps is $K$-quasisymmetric). Using Lemma \ref{lem:ABann}, this new map between the boundary components, together with the restriction of $f_i$ to the remaining boundary components, can be extended to a $(1 + C^\prime \epsilon)$-quasiconformal map which now agrees with the boundary-gluing and defines a map between the glued annuli.
\end{proof}

\section{Conformal limits}

\subsection{Definition}

Let  $\{\Sigma_i\}_{i\geq1}$  be a sequence of marked Riemann surfaces,  and  $\Sigma$ be a Riemann surface with nodes (or punctures) $P$,  such that  $\Sigma \setminus P = S^1 \sqcup S^2\sqcup \cdots S^k$, each a connected and punctured Riemann surface.\\

Assume that for each $i$ we have a  decomposition
\begin{equation}\label{eq:subsurf}
 \Sigma_i = S_i^1 \sqcup S_i^2 \ldots \sqcup S_i^k
 \end{equation}
 into subsurfaces with analytic boundaries,  and a collection of $(1+\epsilon_i)$-quasiconformal embeddings $f_i^j:S_i^j\to S^j$ for each $1\leq j\leq k$ such that :\\
a. $S^j \setminus f_i^j(S_i^j)$ is a disjoint union of punctured disks, \\
b. $\bigcup\limits_{i} f_i^j(S_i^j) = S^j$, and \\
c. $\epsilon_i\to 0$ as $i\to \infty$.\\

We also assume that for any fixed $j$, the images of the maps $f_i^j$ are all isotopic in $\Sigma$. In particular the images of $\partial S_i^j$ are in the same homotopy class (which might enclose points of $P$).  \\

Then $\Sigma$ is said to be a conformal limit of the sequence $\{\Sigma_i\}_{i\geq 1}$.\\

\textbf{Example.} Consider a sequence of hyperbolic surfaces $X_i$  where a separating simple closed curve $\gamma$ is being ``pinched", that is, its length $l(\gamma)\to 0$. Then the conformal limit is the noded Riemann surface as in the figure.\\

\begin{figure}[h]
  \centering
  \includegraphics[scale=0.43]{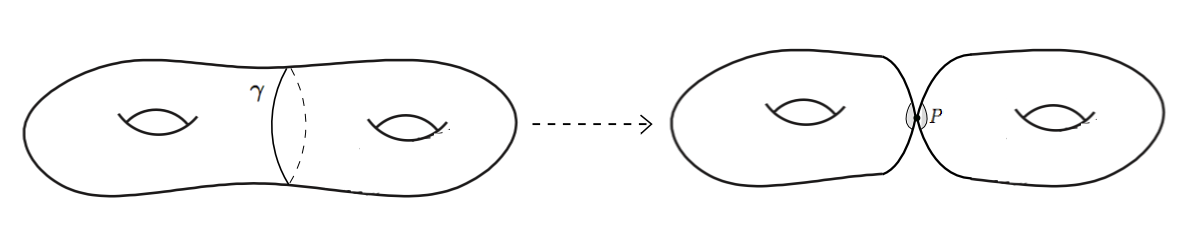}\\
  \caption{The conformal limit of pinching: for $\gamma$ sufficiently short, there is an almost-conformal embedding of its complement to a subsurface away from a small neighborhood of $P$. }
\end{figure}

\begin{lem}[Conformal limit is well-defined]\label{lem:unique} Let $\{\Sigma_i\}_{i\geq 1}$ be a sequence of marked Riemann surfaces such that $\Sigma$ and $\Sigma^\prime$ are conformal limits of some subsequence. Then there exists a conformal homeomorphism $g:\Sigma\to \Sigma^\prime$.
\end{lem}
\begin{proof}
We denote by  $\Sigma^\delta$ the Riemann surface obtained by slightly ``opening" the nodes $P$ on $\Sigma$. More precisely, consider a (fixed) conformal neighborhood $U$ of $P$ that we can identify with a pair of punctured-disks $\mathbb{D}^\ast$, excise the punctured sub-disks of radius $\delta$, and glue along the resulting boundary circles. (This is the usual ``plumbing" construction, see for example \cite{Kra}.) Note that this gluing map on the boundary has quasisymmetry constant $1$.\\

Consider a surface $\Sigma_i$ along the sequence and consider the decomposition (\ref{eq:subsurf}). Recall that the $(1+\epsilon_i)$-quasiconformal embedding $f_i^j:S_i^j \to S^j$ has complement a disjoint union of punctured disks, and as $i\to \infty$, these punctured disks shrink (to the punctures $P$). Hence for sufficiently large $i$, the boundary of the image has an adjacent annular collar $\mathcal{A}$ of large modulus, such that by Lemma \ref{lem:interp} the map $f_i^j$ can be adjusted to a $(1 + C^\prime\epsilon_i)$-quasiconformal map that agrees with $f_i^j$ on $S_i^j\setminus (f_i^j)^{-1}\mathcal{A}$ and maps $\partial S_i^j$ to a round circle. (Note that the quasiconformal map extends to a homeomorphism of the boundary when the latter is an analytic curve.) \\

Piecing these together by Corollary \ref{cor:glueAnn}, we obtain a  $(1 + C^\prime\epsilon_i)$-quasiconformal map  $\bar{f_i}: \Sigma_i\to \Sigma^{\delta_i}$, where $\delta_i\to 0$ and $\epsilon_i\to 0$ as $i\to \infty$.\\ 

Similarly, we have maps $\bar{f_i^\prime}: \Sigma_i \to \Sigma^{\prime \delta_i}$ where the target surface is an ``opening-up" of the nodes on $\Sigma^\prime$. The composition $\bar{f_i^\prime} \circ \bar{f_i}^{-1}:\Sigma^{\delta_i} \to \Sigma^{\prime \delta_i}$ is a $(1+ C^\prime\epsilon_i)$-quasiconformal map that preserves the plumbing curves. Taking a limit as $i\to \infty$, we get the desired conformal homeomorphism $g:\Sigma\to \Sigma^\prime$. \end{proof}

\subsection{Limits of grafting rays}

Consider a grafting ray $\{X_t\}_{t\geq 0}$ determined by a hyperbolic surface $X$ and a geodesic lamination $\lambda$ (see Definition \ref{defn:gray}). In this section we shall introduce the conformal limit, as defined in the previous section, of such a ray. \\

\textbf{Example.} Before the general construction,  consider the case when the lamination is a single simple closed geodesic $\gamma$, The grafting ray $X_t$ comprises longer euclidean cylinders grafted in at $\gamma$. The conformal limit $X_\infty$ in this case is the hyperbolic surface $X\setminus \gamma$ with half-infinite euclidean cylinders glued in at the boundary components (see Figure 6).\\

\begin{figure}[h]
  \centering
  \includegraphics[scale=0.4]{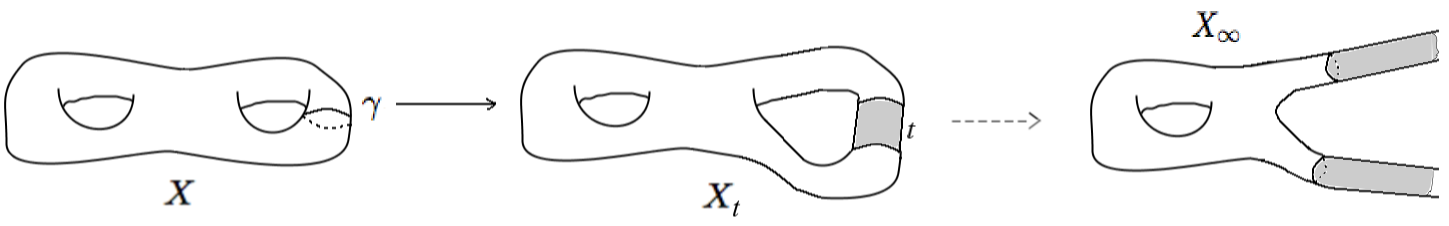}\\
 \caption{An example of the conformal limit of a grafting ray. }
\end{figure}

More generally, we have:

\begin{defn}[$X_\infty$]\label{defn:xinfty} Consider the metric completion $\widehat{X\setminus \lambda}$ of the hyperbolic subsurface in the complement of the lamination. Each boundary component of this completion is either  closed (topologically a circle) or  ``polygonal", comprising a closed chain of bi-infinite hyperbolic geodesics that form ``spikes" (see Figure 8.) Construct $X_\infty$ by gluing in euclidean half-infinite cylinders along the geodesic boundary circles, and euclidean half-planes along the geodesic boundary lines, and  where the gluings are by isometries along the boundary. The resulting (possibly disconnected) surface $X_\infty$ acquires a conformal structure, and a $C^1$-metric that is a hybrid of euclidean and hyperbolic metrics.
\end{defn}

\begin{figure}[h]
  \centering
  \includegraphics[scale=0.5]{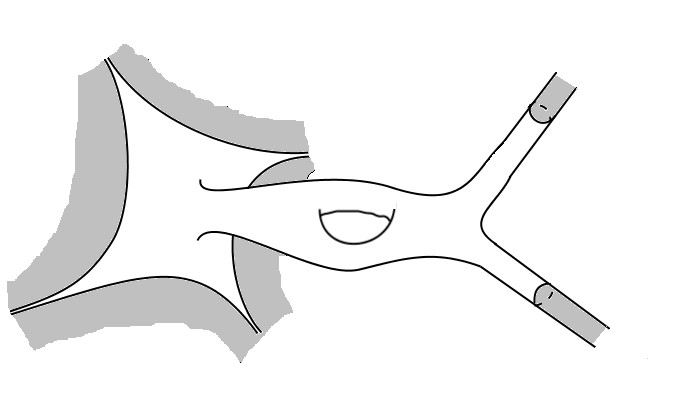}\\
  \caption{Another example of the surface $X_\infty$. The shaded regions are half-planes and half-infinite cylinders.}
\end{figure}

\begin{lem}\label{lem:xlim} $X_\infty$ is the conformal limit of the sequence $\{X_t\}_{t\geq 0}$.
\end{lem}

\begin{proof}
From the above construction we have a decomposition into connected components $X_\infty = S_\infty^1\sqcup \cdots S_\infty^k$, where each component is obtained by attaching half-planes and half-cylinders to the boundary of $S^1,\ldots S^k$.\\

We first construct a subsurface decomposition as in (\ref{eq:subsurf}) for the grafted surface $X_t$:\\

Fix an $\epsilon>0$. Recall the train-track decomposition into rectangles and annuli that carry the lamination (see \S3.1). As before, label the components of the complementary subsurface $X\setminus \mathcal{T}_\epsilon$  by $T_1,T_2,\ldots T_k$. Each $T_i$ is a hyperbolic surface with  boundary components either closed curves, or having geodesic segments that bound  truncated ``spikes"  (for example, a truncated ideal triangle).\\

 We define a ``polygonal piece" $S_t^j$ by expanding each $T_i$:  each geodesic side of $\partial T_i$ has an adjacent rectangle, and each closed circle has an adjacent annulus,  and we append exactly half of those adjacent pieces to $T_i$. These ``expanded" pieces now cover the entire surface $X_t$.  This defines a decomposition into subsurfaces:
 \begin{equation}\label{eq:xdec}
 X_t = S_t^1\sqcup S_t^2 \sqcup \cdots \sqcup S_t^k
 \end{equation}

Moreover, recall from Lemma \ref{lem:widerect} that the euclidean widths of the rectangles increase linearly in $t$, and that for sufficiently large $t$, admit $(1 +\epsilon)$-quasiconformal maps to euclidean rectangles (Lemma \ref{lem:model}). These almost-conformal maps together with an isometry on $T_j$, defines a  a $(1+\epsilon)$-quasiconformal embedding  of the subsurface $S_t^j$ admits to $S^j_\infty$. \\

Consider now a sequence $\epsilon_i\to 0$ as $i\to \infty$. The heights of the rectangles in the above decompositions increase (Lemma \ref{lem:thinrect}) and we can choose a corresponding sequence of $t_i$-s where $t_i\to \infty$ such that an almost-conformal embedding as above exists, and their images exhaust each component as $i\to \infty$.\\

This satisfies the conditions of the definition of a conformal limit in \S4.1.

\end{proof}

\subsection{Limits of Teichm\"{u}ller rays}

In this section we define a conformal limit for a Teichm\"{u}ller ray $\{Y_t\}_{t\geq 0}$ determined by a basepoint $Y\in \mathcal{T}_g$ and a holomorphic quadratic differential $q \in Q(Y)$. As we shall see, this limiting surface is equipped with a singular flat metric of infinite area, which the singular-flat surfaces along the ray converge to, when suitably rescaled. Hence this is in fact a metric  convergence (stronger than the one in the previous section) and can be thought of as a Gromov-Hausdorff limit of the surfaces along the Teichm\"{u}ller ray.\\

\textbf{A warm-up example.} Consider the case when $q$ is a Jenkins-Strebel differential, such that all its vertical leaves are closed and foliate a single cylinder. Its boundary after identifications forms an embedded metric graph on the surface.  In this case the flat cylinder lengthens along the Teichm\"{u}ller ray, and the limit ``as seen from" the metric graph consists of half-infinite cylinders glued along the graph (see Figure 8).

\begin{figure}[h]
  \centering
  \includegraphics[scale=0.8]{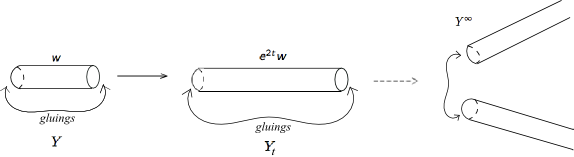}\\
  \caption{A conformal limit along a Strebel ray.}
\end{figure}

\subsubsection{Vertical graphs}

\begin{defn}\label{defn:vgraph}  The \textit{vertical graph} of saddle connections $\mathcal{V}(q)$ for a quadratic differential $q$ is the (possibly disconnected) graph embedded in $X$ that consists of vertices that are the zeroes of $q$, and edges that are vertical segments between zeroes (i.e, the saddle-connections).
\end{defn}

\textit{Remark.} Generically, there are no saddle connections and the zeroes are simple, so $\mathcal{V}(q)$ consist of a collection of $4g-4$ points on the surface.\\

\begin{figure}[h]
  \centering
  \includegraphics[scale=0.45]{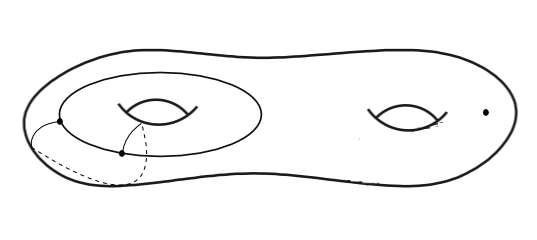}\\
  \caption{The vertical graph of saddle connections for a holomorphic quadratic differential on a genus $2$ surface (see \S5.3.4).}
\end{figure}

In what follows we shall also consider:

\begin{defn}\label{defn:avgraph} The \textit{appended} vertical graph $\mathcal{V}_L(q)$ obtained by appending vertical segments of length $L$ from the remaining vertical (non-saddle-connection) leaves emanating from the vertices of $\mathcal{V}(q)$\end{defn}

\textit{Remark.} The appended graph $\mathcal{V}_L(q)$ embeds on the surface, for any length $L$. Moreover, by the classification of the trajectory-structure for a compact surface  (see \cite{Streb}), the graph is dense on the surface as $L\to \infty$.\\

Consider a $\delta$-neighborhood $\mathcal{N}$ (a  slight ``thickening") of the embedded  $\mathcal{V}_L(q)$ on the surface. 
 
\begin{defn}\label{defn:side} A \textit{side} is a subgraph of $\mathcal{V}_L(q)$ adjacent to the same component of $\mathcal{N}\setminus \mathcal{V}_L(q)$.
\end{defn}

\subsubsection{Polygonal decomposition}

We can define a decomposition of the surface into polygons based on  the appended vertical graph $\mathcal{V}_L(q)$ as follows:\\

Let the number of sides of $\mathcal{V}_L(q)$ be $s$. Consider a ``horizontal $(\delta_1,\delta_2,\ldots,\delta_s)$-collar" of such a graph obtained by appending an adjacent rectangular region of horizontal width $\delta_i$ along the $i$-th side if it is not a cycle. If side forms a cycle, we instead append an adjacent annular region of width $\delta_i$. For $L$ sufficiently large, and any positive tuple of widths, this  adjacent metric collar covers the entire surface (see the remark above). \\

For any such $L$, consider the ``maximal" horizontal collar of $\mathcal{V}_L(q)$ (an appropriate choice $\delta_1,\ldots \delta_s$ ) whose interior embeds, and whose closure covers the surface. \\

This defines a polygonal decomposition: for each component of $\mathcal{V}_L(q)$ one has an embedded region (a union of rectangles and annuli) containing it with a polygonal boundary of alternating horizontal and vertical sides, or vertical closed curves. We denote these regions by $Y^1,Y^2, \ldots Y^k$. (Here $k$ is the number of connected components of  $\mathcal{V}_L(q)$, a number bounded above by the topological complexity of the surface.)\\

\textit{Remark.} From the polygonal decomposition above one can construct a weighted train-track corresponding to the vertical foliation (the rectangles of each polygonal piece form the branches). This is related to the train-track construction for quadratic differentials in a given strata in \S3 of \cite{Hamen}.

\subsubsection{The conformal limit}

Consider the (possibly disconnected) metric graph $\mathcal{V}_\infty(q)$ obtained as a limit of the appended graphs (Definition \ref{defn:avgraph}) as the length of the appended horizontal segment $L\to \infty$.  (We consider these as abstract metric graphs without reference to their embedding on the surface.) \\

Let $\mathcal{V}^1, \mathcal{V}^2 , \ldots,  \mathcal{V}^k$ be the connected components of this graph. 

\begin{defn}[$Y_\infty$]\label{defn:ylim}  The surface $Y_\infty$ is the complete singular-flat surface of infinite area obtained by attaching euclidean half-planes along each side of $\mathcal{V}_\infty(q)$ that is not a cycle, and attaching half-infinite euclidean cylinders along each side that is a cycle, by isometries along their boundaries. This is a punctured Riemann surface with $k$ connected components $Y^1,Y^2,\ldots Y^k$, where $Y^j$ has a metric spine $\mathcal{V}^j$.
\end{defn}

\textit{Remark.} This limiting singular-flat surface of infinite area is a ``half-plane surface" in the sense of \cite{Gup25} (we recall this notion in the Appendix). In the generic case (see the remark following Definition \ref{defn:vgraph}) we get a collection of $4g-4$ copies of the complex plane $\mathbb{C}$ each equipped with the quadratic differential metric induced by $zdz^2$.

\begin{lem} $Y_\infty$ is the conformal limit of the Teichm\"{u}ller ray $Y_t$.
\end{lem}
\begin{proof}
For a fixed $L>0$ consider the polygonal decomposition $Y = Y^1 \sqcup Y^2 \sqcup \cdots \sqcup Y^k$ as in the previous section.  Each piece can be thought of as being built by rectangles glued along an appended vertical graph $\mathcal{V}_L(q)$. Along the Teichm\"{u}ller ray, these rectangles widen in the  horizontal direction, and defines a polygonal decomposition (with the same combinatorial pattern of gluing):  
\begin{equation}
Y_t = Y^1_t \sqcup Y^2_t \sqcup \cdots \sqcup Y^k_t
\end{equation}
It is clear from the construction that each $Y^j_t$ embeds isometrically  in $Y^j$  via an embedding of the rectangles into half-planes (see Definition \ref{defn:ylim}). The complement of this embedded image is a planar end  $\mathcal{P}$ (see Definition \ref{defn:pend}) and, in particular, a punctured disk.\\

Consider a sequence $L_i\to \infty$, and consider the sequence of decompositions as above.  For each $i$, we can choose $t_i$ large enough, such that the rectangles have width greater than $L_i$, and since they have lengths more than $2L_i$ also (recall we append vertical segments of length $L_i$ in the appended graph).\\

The isometric embeddings of $Y^1_{t_i},\ldots Y^k_{t_i}$ into $Y_\infty$ include rectangles of height and width $2L_i$ on each half-plane, as above. The sequence of these embedded images then exhaust $Y_\infty$ as $i\to \infty$.\\

These satisfy the definition of $Y_\infty$ being the conformal limit (see \S4.1 - note that the embeddings, being isometries, are in fact \textit{conformal} maps, and not just almost-conformal).\end{proof}

\subsubsection{An example}

Consider a flat torus obtained as a translation surface, identifying sides of a square embedded in $\mathbb{R}^2$ with irrational slope. Introduce a vertical slit and glue the resulting segments by an interval-exchange map.

\begin{figure}[h]
  \centering
  \includegraphics[scale=0.7]{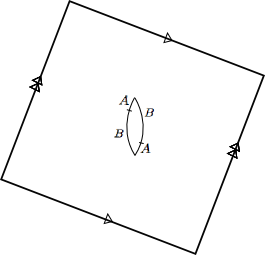}\\
  \caption{After the sides of the vertical slit on the torus are glued as labelled, one gets a genus $2$ surface. The vertical graph of the induced quadratic differential is the one in Figure 9, and the conformal limit is a half-plane surface with a gluing of two half-planes, as in Figure 13.}
\end{figure}

 Then the Teichm\"{u}ller ray $Y_t$  is obtained by the action of the diagonal subgroup of $SL_2(\mathbb{R})$, acting by linear maps on $\mathbb{R}^2$. For a sequence $t_i\to \infty$, we can take suitable rescalings of $Y_{t_i}$ and get as a limit a singular flat surface obtained by gluing two euclidean half-planes by an interval exchange of their boundaries.

\subsubsection{Some remarks} We remark that the metric residue (see Definition \ref{defn:pend}) of the limit $Y_\infty$ of a Teichm\"{u}ller ray can be determined from the vertical graph $\mathcal{V}(q)$ by taking the alternating sum of the lengths of the cycle of sides corresponding to the ``end" (it is independent of the appended ``feelers" of length $L\to \infty$). In \S7.1 we develop this notion for limits of grafting rays.\\

Conversely, given a (possibly disconnected) generalized half-plane surface (see Definition \ref{defn:ghp}) $Z$ with a pairing of ends with the same order and residue, there exists a Teichm\"{u}ller ray $\{Z_t\}_{t\geq 0}$ with conformal limit $Z$. One can construct this by considering a truncation of $Z$ (see Definition \ref{defn:ytrunc}) and gluing together the polygonal boundaries by isometries (which is possible as the residues match) to obtain a closed surface $Z_0$. See  \S7.4 for such a construction. This truncation (in particular the lengths of the horizontal edges) are chosen to be in irrational ratios such that the appended ``feelers" are dense on the surface as $L\to \infty$ (and do not form additional vertical saddle-connections).\\

Details of this, and a fuller characterization of the singular-flat surfaces that appear as conformal limits of Teichm\"{u}ller rays, shall be addressed in future work.

\section{Strategy of the proof}

The proof of Theorem \ref{thm:thm1} following the outline in this section is carried out in \S7. \\

As before we fix a hyperbolic surface $X$ and measured lamination $\lambda$ and consider the grafting ray $X_t = gr_{t\lambda}X$. Our task is to find a Teichm\"{u}ller ray $Y_t$ such that under appropriate parametrization, the Teichm\"{u}ller distance between the rays tends to zero.\\

We first briefly recall the strategy for the case when $\lambda$ is a multicurve which, along with the arational case (see \S3.4), was dealt with in \cite{Gup1}. As we outline below, this generalizes to the remaining case of non-filling laminations, dealt with in this paper. 

\subsection{Multicurve case}

As described in \S4.2 along the grafting ray the surfaces acquire increasingly long euclidean cylinders along the geodesic representatives of the curves, and one considers the conformal limit $X_\infty$ that has half-infinite euclidean cylinders inserted at the boundary components of $X\setminus \lambda$. \\

A theorem of Strebel then shows the existence of a certain meromorphic quadratic differential with poles of order two on $X_\infty$. This produces a singular flat surface $Y_\infty$ comprising half-infinite euclidean cylinders, together with a conformal map $g:X_\infty \to Y_\infty$. Suitably truncating the cylinders on $X_\infty$ and $Y_\infty$, adjusting $g$ to an almost-conformal map between them, and gluing the truncations, produces for all sufficiently large $t$, an almost-conformal map between $X_t$ and a surface along a Teichm\"{u}ller ray $\{Y_t\}$  that limits to $Y_\infty$. Details are in \S5 of \cite{Gup1}.\\

In the more general case handled in this paper, one needs Theorem \ref{thm:hpd} (see the Appendix), which is the appropriate generalization of the theorem of Strebel mentioned above.

\subsection{Proof outline}

Fix an $\epsilon>0$.  The proof of Theorem \ref{thm:thm1} will be complete if one can show that for all sufficiently large $t$, there exists a $(1+\epsilon)$-quasiconformal map:
\begin{equation*}
f:X_t\to Y_t
\end{equation*}
where $X_t$ is the surface along the grafting ray determined by $(X,\lambda)$, and $Y_t$ lies along a Teichm\"{u}ller ray determined by $\lambda$. \\

\begin{figure}
  \centering
  \includegraphics[scale=0.43]{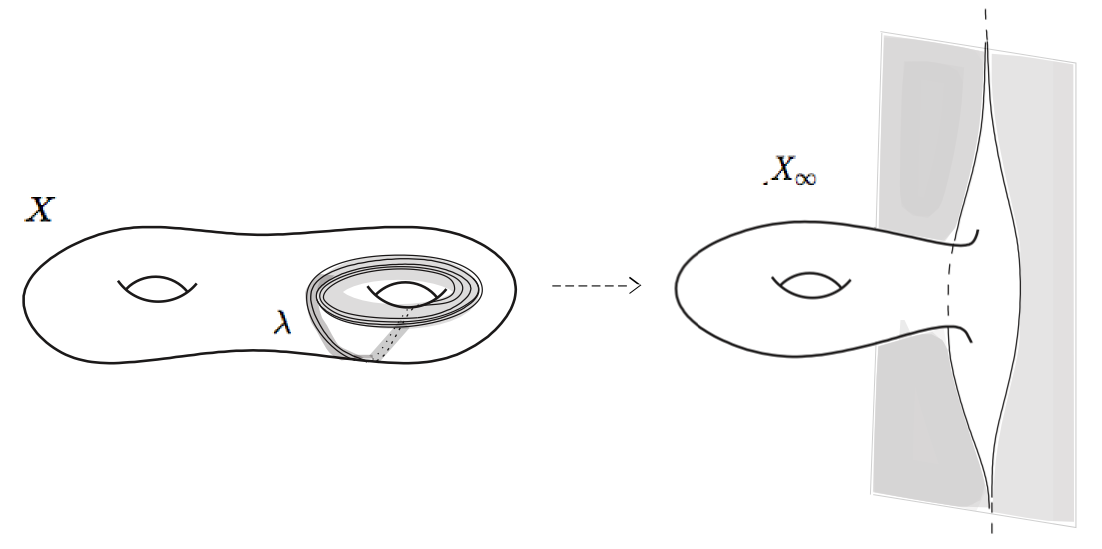}\\
  \caption[The conformal limit for a general lamination.]{An example of the conformal limit $X_\infty$ for a minimal lamination $\lambda$, obtained by attaching euclidean half-planes along the boundary of the completion of $X\setminus \lambda$.}
\end{figure}

Recall from \S5.2 that there is a conformal limit of the grafting ray:
\begin{equation}\label{eq:xinfdec}
X_\infty = S_\infty^1 \sqcup S_\infty^2 \sqcup \cdots \sqcup  S_\infty^k
\end{equation}
where each $S_\infty^j$ is a surface obtained by appending half-planes and half-infinite cylinders to the boundary of a connected component of $\widehat{X\setminus \lambda}$.\\

Our strategy is outlined as follows: \newline

\textit{Step 1.} By specifying an appropriate meromorphic quadratic differential on each $S_\infty^1,\ldots S_\infty^k$ using Theorem \ref{thm:hpd}, we find a singular flat surface
\begin{center}
$Y_\infty=  Y_\infty^1 \sqcup Y_\infty^2 \sqcup \cdots \sqcup  Y_\infty^k$ 
\end{center}
conformally equivalent to $X_\infty$ (the singular flat metric is the one induced by the differential). The ``local data" (eg. orders and residues) at the poles of the meromorphic quadratic differential are prescribed according to the geometry of the ``ends" of each $S_\infty^j$. The fact that the underlying marked Riemann surfaces are identical then gives conformal homeomorphisms $g^j:S_\infty^j \to Y_\infty^j$ each homotopic to the identity map.  \\

\textit{Step 2.} By the quasiconformal interpolation of Lemma \ref{lem:interp}, each conformal map $g^j$ of \textit{Step 1} is adjusted to produce an  $(1+\epsilon)$-quasiconformal map $h^j: S_\infty^j \to Y_\infty^j$ that is ``almost the identity map"  near the ends.  In particular, for any  ``truncations" of those infinite-area surfaces at sufficiently large height, the map $h^j$ preserves, and is almost-isometric on, the resulting polygonal boundaries.\\

\begin{figure}
  \centering
  \includegraphics[scale=0.33]{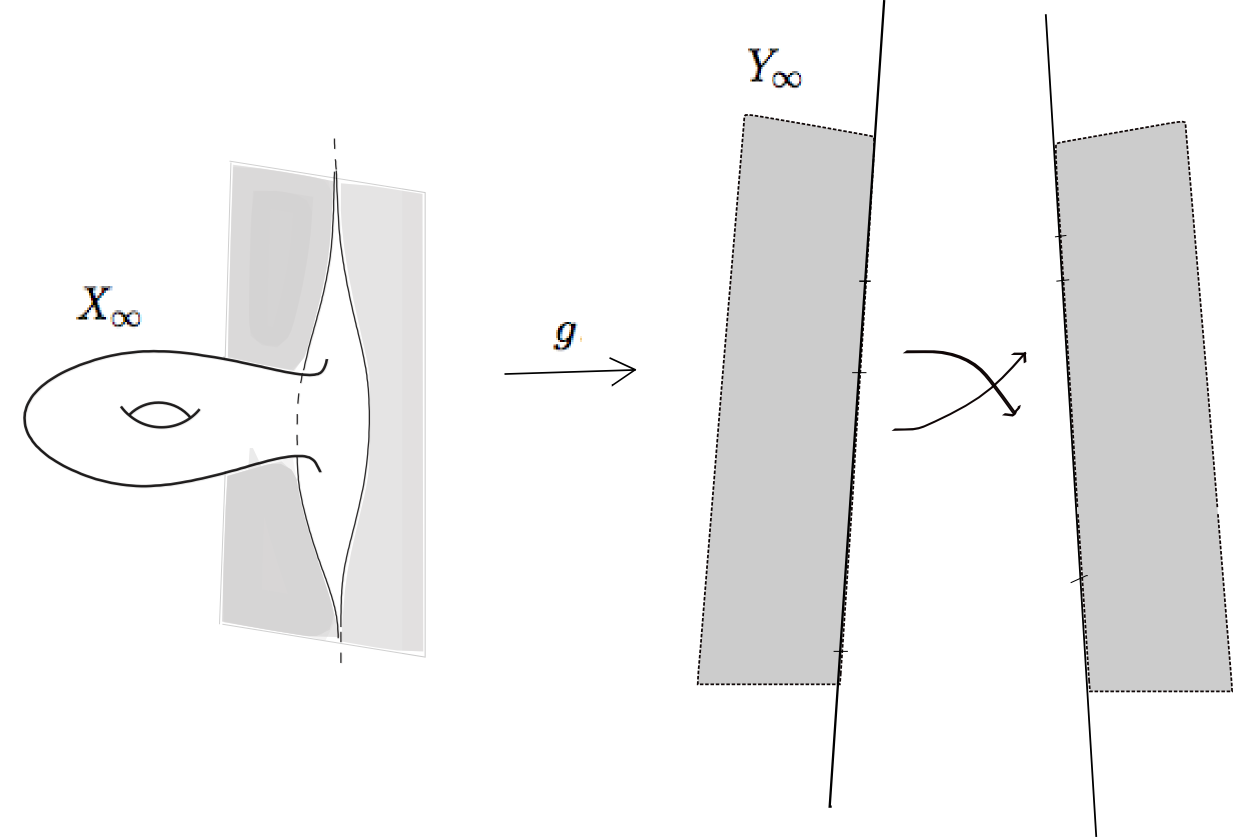}\\
  \caption[Half-plane surface]{ In \textit{Step 1}  one finds a conformal homeomorphism $g$ from the infinitely grafted surface to a generalized half-plane surface $Y_\infty$.}
\end{figure}

\textit{Step 3.} As in \S5.2 the grafted surface at time $t$ has the decomposition:
\begin{equation*}
X_t = S_t^1 \sqcup S_t^2 \sqcup \cdots \sqcup  S_t^k
\end{equation*}
where each $S_t^j$ is obtained by appending halves of adjacent branches of a train-track $\mathcal{T}$ carrying the lamination. The euclidean part of these branches can be glued up to produce singular flat surfaces $Y_{t}^1,\ldots Y_{t}^k$ that embed isometrically in $Y_\infty^1, \ldots Y_\infty^k$ respectively. Gluing these in the pattern determined by that of the $S_{t}^j$-s then produces a surface $Y_{t}$ which lies on a common Teichm\"{u}ller ray (independent of the choice of train-track $\mathcal{T}$).\newline

\textit{Step 4.} By Lemma \ref{lem:model}, for sufficiently large $t$ the surfaces $S_t^1, S_t^2, \ldots S_t^k$  admit almost-conformal embeddings into $S_\infty^1,\ldots S_\infty^k$.  If the train-track chosen in \textit{Step 3} had ``sufficiently tall" branches, the boundary of these embeddings lie far out an end, and the almost-conformal maps of \textit{Step 2} then further map them into $Y_t^1,\ldots ,Y_t^k$ respectively. By the quasiconformal extension Lemma \ref{lem:ext}, the almost-conformal maps in the above composition can be adjusted along the boundaries such that they fit continuously to produce an almost-conformal map $f:X_t \to Y_t$ between the glued-up subsurfaces.

\section{Proof of Theorem 1.1}

We shall follow the outline in the previous section, and refer to that for notation and the setup.\\

We begin by associating a non-negative real number to each topological end of  the infinitely-grafted surface $X_\infty$. 
\subsection{End data}

Let $S^1,S^2,\ldots S^k$ be the components of the metric completion $\widehat{X\setminus \lambda}$.\\

Recall that for any $1\leq j\leq k$, a component of the boundary of $S^j$ is  either \textit{closed} (a geodesic circle) or is \textit{polygonal} (a geodesic circle with ``spikes" - this is also called a ``crown" in \cite{CassBl}).\\

Consider a polygonal boundary consisting of a cyclically ordered collection of $n$ bi-infinite geodesics $\{\gamma_1,\gamma_2,\ldots \gamma_n\}$.\\

Choose basepoints $p_i\in \gamma_i$ for each $1\leq i\leq n$. This choice gives the following notion of ``height" on each half-plane adjacent to this polygonal boundary:  

\begin{defn}[Heights]\label{defn:height}
For any point on the half-plane, follow the horizontal line until it hits one of the $\gamma_i$-s. The hyperbolic distance of this point from $p_i$ (measured with sign) is the \textit{height} of $p$.
\end{defn}

\begin{defn}[Polygonal boundary residue]\label{defn:polyd} For each polygonal boundary as above, we associate a non-negative real number $c$ as follows:\\
(1) If $n$ is odd, $c=0$.\\
(2)  If $n$ is even,  choose horocyclic leaves sufficiently far out in the cusped regions between the $\gamma_i$-s. Suppose they connect a point on $\gamma_i$ at height $R_i$ with a point on $\gamma_{i+1}$ at height $L_{i+1}$. If one cuts along these arcs to truncate the spikes, we get a boundary consisting of alternating geodesic and horocyclic sides,  where the geodesic side lying on $\gamma_i$ has length $R_i-L_i$. Then define $c =  \lvert \sum\limits_{i=1}^n (-1)^{i+1} (R_i-L_i)\rvert $.\\
We call this the \textit{residue} for the polygonal boundary, in analogy with the ``metric residue" of a planar end (see Definition \ref{defn:pend}).
\end{defn}

\begin{figure*}[h]
  \centering
  \includegraphics[scale=0.45]{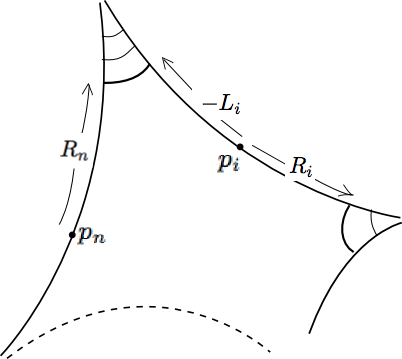}\\
\end{figure*}

\textit{Remark.} It is easy to see that the definition in (2) above is independent of the choice of base-points: a different choice of $p_i$ will increase (or decrease) both $L_i$ and $R_i$ by the same amount, and the difference remains the same. Moroever, the definition is independent of the choice of horocyclic leaves of the truncation, as the next lemma shows.

\begin{lem}\label{lem:eolem} Let the residue for a polygonal boundary be $C$. Then for any $H$  sufficiently large, there is a choice of basepoints, and a choice of horocyclic arcs  such that the set of lengths of the $n$ geodesic sides after truncation is $\{H,H,\ldots H, H+C\}$.
\end{lem}
\begin{proof}
Fix a basepoint $p_1$ on $\gamma_1$. For any $H$ sufficiently large, there is a horocyclic leaf at distance $H$ from $p_1$ on $\gamma_1$. Following that leaf, we get to $\gamma_2$, and we pick a basepoint $p_2\in \gamma_2$ such that the horocylic leaf was at $-H/2$ on $\gamma_2$.  Here $H$ is chosen sufficiently large so that there is a horocylic leaf at $H/2$ on $\gamma_2$, and we keep following horocylic leaves and picking basepoints on each successive $\gamma_i$  which form the midpoints of segments of length $H$, until we come back to a point $p^*$ on $\gamma_1$ (the end point of the horocylic leaf at $H/2$ on $\gamma_n$).\\
 Let $C$ be the distance between $p_1$ and $p^*$ on $\gamma_1$. When $n$ is odd, this distance can be reduced to $0$ by moving our initial choice of basepoint $p_1$ ($p^*$ moves in a direction opposite to $p_1$). When $n$ is even, this distance is independent (and equal to $C$) for any initial choice of basepoint $p_1$. In particular, if we chose the basepoint on $\gamma_n$ to be the midpoint of a segment of length $(H+ C)/2$ instead of $H/2$, the point $p^*$ coincides with $p_1$.
 \end{proof}

Now consider the conformal limit  as in (\ref{eq:xinfdec}). For each $1\leq j\leq k$ the surface $S_\infty^j$ has either cylindrical ends, corresponding to the closed boundary components of $S^j$, or flaring ``planar" ends corresponding to the polygonal boundaries. 

\begin{defn}[End data]\label{defn:endd}For each end, we can associate a \textit{residue} which for a closed boundary equals its length, and is a non-negative number as in Definition \ref{defn:polyd} for a polygonal boundary.
\end{defn}

\begin{defn}[$H$-truncation]\label{defn:trunc} Let $C$ be the residue of an end of  $S_\infty^j$,  for some $1\leq j\leq k$, as defined above. A \textit{truncation of the end} at height $H$ will be the surface obtained as follows: for a  polygonal boundary of $S^j$ choose a truncation of the spikes such that the geodesic sides have lengths $\{H,H,\ldots H+C\}$ where $C$ is the residue for the end (see Lemma \ref{lem:eolem}), and append euclidean rectangles of horizontal widths $H/2$ along each. For each closed boundary component, we append a euclidean annulus of width $H/2$.\\

Note that the complement of the truncation of an end is a punctured disk.
\end{defn}

\textit{Remark.} A similar terminology can be adopted for a truncation of a half-plane surface (see also Definition \ref{defn:ytrunc}) - an $H$-truncation shall be the gluings of rectangles of width $H/2$ and heights $H$ (except one of $H+C$) along its truncated metric spine.

\subsection{Step 1: The surface $Y_\infty$}
The task at hand is to define an appropriate singular flat surface $Y_\infty$  that is conformally equivalent to $X_\infty$, that shall be the conformal limit  of the asymptotic Teichm\"{u}ller ray.\\

For each $1\leq j\leq k$, consider the surface $S_\infty ^j$ (see (\ref{eq:xinfdec})). This is a Riemann surface with punctures $p_1,\ldots p_m$ (corresponding to the planar or cylindrical ends). As in the previous section, these have residues $c_1,\ldots c_m$ and one can choose punctured-disk neighborhoods $U_1,\ldots U_m$ which are the complements of truncations of the ends at some choice of heights.\\

We shall now equip $S_\infty^j$ with a meromorphic quadratic differential that induces a conformally equivalent singular flat metric with a ``(generalized) half-plane structure" (see Appendix B):\\

Namely, by Theorem \ref{thm:hpd} there exists a singular flat surface $Y_\infty^j$ that is conformally equivalent to $S_\infty^j$ and has planar ends  corresponding to the punctures  with metric residues $c_1,\ldots c_m$, such that the conformal homeomorphism:
\begin{equation*}
g^j:S_\infty^j \to Y_\infty^j
\end{equation*}
preserves the punctures, and the \textit{leading order terms} of the generalized half-plane differential (see Definition \ref{defn:locdat}) on $Y_\infty^j$ with respect to the coordinate neighborhoods $V_1=g^j(U_1),V_2=g^j(U_2)\ldots, V_m=g^j(U_m)$ are all equal to $1$.\\

\textit{Remark.} The Gauss-Bonnet theorem rules out the possibility of the exceptional cases of Theorem \ref{thm:hpd}: namely, there cannot be a hyperbolic surface with two infinite geodesics bounding a ``bigon" , nor can their be two distinct closed geodesics bounding an annulus.  \\

We define
\begin{equation*}
Y_\infty = Y_\infty^1 \sqcup Y_\infty^2 \sqcup \cdots \sqcup Y_\infty^k
\end{equation*}
and the union of the maps above gives a conformal homeomorphism
\begin{equation}\label{eq:mapg}
g:X_\infty \to Y_\infty
\end{equation}

The choice of leading order terms above implies the following (see Lemma \ref{lem:lot}):

\begin{lem}\label{lem:gderiv}The derivative of $g$ at each pole with respect to uniformizing coordinates is $1$. That is, for any $j$, let the conformal maps $\phi:U_j\to \mathbb{D}$ and $\psi: V_j\to \mathbb{C}$ take $p$ and $\infty$ respectively, to $0$. Then $\lvert \left( \psi\circ g\circ\phi^{-1}\right)^\prime(0)\rvert =1$.
\end{lem}

\subsection{Step 2: Adjusting the map $g$}

Consider the surfaces $S_\infty^j$ and $Y_\infty^j$ as in the previous section, for some $1\leq j\leq k$.\\\

In what follows, we shall denote their $H$-truncations (see Definition \ref{defn:trunc} and the following remark)  by $S_H^j$ and $Y_H^j$, and their complements by $Q_H^j$ and $P_H^j$ respectively. \\

We shall assume that  $\partial S^j$ has exactly one polygonal boundary: the case of closed boundary has been handled before (see \cite{Gup1}, and \S6.1 for an outline), and when there are several boundary components the  following arguments are identical, except  one needs to keep track of each component with an additional cumbersome index.\\

By this assumption  $Q_H^j$ and $P_H^j$ are topologically punctured disks, and  $P_H^j$ is a planar end in the sense of Definition \ref{defn:pend}.\\

\begin{figure}[h]
  \centering
  \includegraphics[scale=0.5]{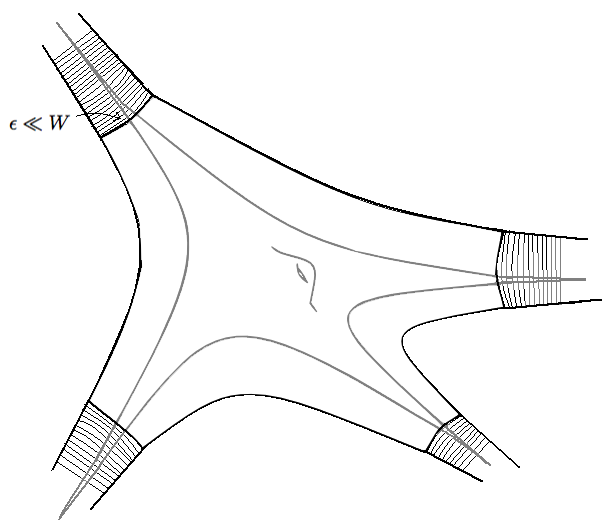}\\
  \caption{This shows $Q_t^j $ with its polygonal boundary, and the infinite strips adjoining each horizontal edge of width $W$. In Lemma \ref{lem:fj} the map $f^j$ collapses the $\epsilon$-thin hyperbolic ``spikes"  and is an isometry on the half-planes in the complement. }
\end{figure}

Recall from \S6.2 that we have fixed an $\epsilon>0$.

\begin{lem}\label{lem:fj}
There exists an $H_0>0$  such that there is a $(1+\epsilon)$-quasiconformal map 
\begin{equation*}
f^j:Q_{H_0}^j \to P_{H_0}^j
\end{equation*}
that is height-preserving, and $(\epsilon,\epsilon)$-almost-isometric on any horizontal segment in the domain.
\end{lem}

\begin{proof}
Suppose $S^j$ has a polygonal end with $n$ geodesic lines as boundary components, so $S_\infty^j$ has $n$ half-planes attached along them. The subset $Q_{H_0}^j$ has $n$ euclidean regions, each isometric to a ``notched" half-plane. These regions are arranged in a cycle with adjacent ones glued to each other by a thin hyperbolic ``spike" of width less than $\epsilon$ (for sufficiently large $t$). This euclidean-and-hyperbolic metric on $Q_{H_0}^j$ is $C^1$.\newline

The map $h^j$ is obtained essentially by collapsing these spikes. Divide $Q_{H_0}^j$ and $P_{H_0}^j$ into vertical strips and half-planes as shown in Figure 14. Every vertical edge of $\partial Q_{H_0}^j$ is part of a vertical line bounding a half-plane which is mapped isometrically to the corresponding half-plane in $P_{H_0}^j$. The remaining infinite strips adjoining each horizontal edge is mapped by a height preserving affine map to the corresponding strips in $P_{H_0}^j$.\\

Assume that $H_0$ is sufficiently large such that each of these horizontal edges have euclidean length $W=2H_0\gg \epsilon$. The horizontal stretch factor for the affine map of the strips  is then close to $1$, and since it is height-preserving it preserves vertical distances.  Being height-preserving the affine map agrees with the map on the half-planes previously described, and thus we have a $C^1$ map $h^j$  that is close to being an isometry, and is hence almost-conformal. It is easy to check that a stretch map between two sufficiently long intervals whose lengths differ by $\epsilon$, is an $(\epsilon, \epsilon)$-almost-isometry in the sense of Definition \ref{defn:ecisom}. \end{proof}

Using the quasiconformal interpolation of Lemma \ref{lem:interp}, we now have:

\begin{prop}\label{prop:adg}
There exists $H_1\gg H_0$, such that there is a $(1+C\epsilon)$-quasiconformal map 
\begin{equation}\label{eq:hj}
h^j: S_\infty^j\to Y_\infty^j
\end{equation}
that agrees with $f^j$ on $Q_{H_1}^j$ and $g^j$ on $S_{H_0}^j$.\\
(Here $C>0$ is a universal constant.)
\end{prop}
\begin{proof}

By Lemma \ref{lem:fj}, there is a $(1+\epsilon)$-quasiconformal map $f^j: Q_{H_0}^j\to P_{H_0}^j$.\\

We conformally identify $Q_{H_0}^j$ and $P_{H_0}^j$ with $\mathbb{D}$ via uniformizing maps that take $\infty$ to $0$. By  restricting $g^j$ we have the conformal embedding $\bar{g^j}: B_r \to g(B_r)\subset \mathbb{D}$. Here, via the above identifications,  the domain $B_r$ is a subset of $Q_{H_0}^j$, and the radius $r$ is chosen small enough such that its image under $g^j$ is a subset of  $P_{H_0}^j$. All these maps preserve $0\in \mathbb{D}$, and by Lemma \ref{lem:gderiv},  the map $\bar{g^j}$ has derivative $1$ at $0$. \\

Applying Lemma \ref{lem:interp} we then have a $(1+C\epsilon)$-quasiconformal map $F^j:B_r \to g(B_r)$ that agrees with $\bar{g^j}$ on the boundary and restricts to $f^j$ on a smaller subdisk $B_{r^\prime}$.\\

Choose $H_1 \gg H_0$ such that $B_{r^\prime}$ contains $Q_{H_1}^j$.\\


The map $h^j: S_\infty^j\to Y_\infty^j$ is now defined to be the one that  restricts to  $F^j$ on $B_r$, and to  $g^j$ to $S_{\infty}^j \setminus B_r$.

\end{proof}

\textit{Remark.} The property of $f^j$ in Lemma \ref{lem:fj} being height-preserving implies that for any $H>H_1$ the map $h^j$ of (\ref{eq:hj}) maps the truncation $S_{H}^j \subset S_\infty^j$ to the corresponding truncation $Y_H^j\subset Y_\infty^j$.

\subsection{Step 3: Defining the Teichm\"{u}ller ray}

For a sufficiently small $\epsilon^\prime>0$ (the choice of which shall be clarified at the end of this section) , consider as in \S3.1 a train-track neighborhood $\mathcal{T}_{\epsilon^\prime}$ of the lamination $\lambda$.\\

Recall from \S5.2 that there is the corresponding decomposition of the grafted surface at time $t$:
\begin{equation}\label{eq:xtdec}
X_t = S_t^1 \sqcup S_t^2 \sqcup \cdots \sqcup  S_t^k
\end{equation}
where each $S_t^j$ is obtained by appending halves of adjacent branches of $\mathcal{T}_{\epsilon^\prime}$ adjacent to the components of $X\setminus \mathcal{T}_{\epsilon^\prime}$.\\

The goal of this subsection is to define the Teichm\"{u}ller ray in the direction of $\lambda$. This is done by piecing together certain truncations of the (generalized) half-plane surfaces $Y_\infty^1,\ldots Y_\infty^k$, in the same pattern as the $S_t^j$-s in (\ref{eq:xtdec}).

\begin{defn}\label{defn:ytrunc} A \textit{truncation} of a generalized half-plane surface $Y$ is the singular flat surface with ``polygonal" boundary obtained by truncating the infinite edges of the metric spine of $Y$, and gluing euclidean rectangles along each non-closed side, and euclidean annuli along each closed side (i.e, considering a ``horizontal collar" as in \S5.3.2).
\end{defn}

For any $1\leq j\leq k$  consider the surface $S_t^j$ as in (\ref{eq:xtdec}). Consider a truncation of $Y_\infty^j$ that involves rectangles and annuli of same vertical heights, and same \textit{euclidean} widths, glued in the same pattern as $S_t^j$. In particular, the resulting singular flat surface $Y_t^j$ is homeomorphic to $S_t^j$. The fact that the metric residue (see Defn \ref{defn:pend}) at punctures of $Y_\infty^j$ are equal to the residues of the ends of $S_\infty^j$ ensures  that the euclidean rectangles glue up to give a polygonal boundary of $Y_t^j$ with sides corresponding to those of $\partial S_t^j$.

\begin{defn}[$Y_t$]\label{defn:yt} The closed singular-flat surface $Y_t$ is obtained by gluing the singular flat surfaces $Y_t^1,Y_t^2,\ldots Y_t^k$ along their boundaries according to the gluing of $S_t^1,\ldots S_t^j$ on $X_t$:\\
Namely, any point of the boundary of $S_t^j$ has a vertical coordinate determined by its height (Definition \ref{defn:height}) and a horizontal coordinate equal to the euclidean distance from the geodesic boundary of $S^j$ (this is zero for a point on the horocyclic sides).  Similarly we have horizontal and vertical coordinates on $\partial Y_t^j$ determined by the euclidean metric (the metric spine has zero horizontal coordinate). The gluing maps between the $Y_t^j$ in these coordinates is then identical to those between the $S_t^j$ on $X_t$.  It is not hard to see that these gluing maps between $Y_t^j$-s are isometries in the singular flat metric, and one obtains a closed singular-flat surface $Y_t$.
\end{defn}

\medskip
Summarizing, we have:\\
\begin{equation*}
Y_t = Y_t^1 \sqcup Y_t^2 \sqcup \cdots \sqcup Y_t^k
\end{equation*}

\medskip
\begin{lem}\label{lem:Yt} For all $t$ the surfaces $Y_t$ lie along a Teichm\"{u}ller geodesic ray that is independent of the train-track $\mathcal{T}_{\epsilon^\prime}$ chosen at the beginning of the section.
\end{lem}
\begin{proof}
By construction of the surface $Y_t$, the combinatorial weighted ``train-track" carrying the vertical foliation is identical (as a marked graph) to that for the measured lamination $t\lambda$ on $X_t$ (a smaller choice of $\epsilon^\prime$ yields a train-track with more ``splitting").  This implies that the  vertical foliation on $Y_t$ is measure-equivalent to $t\lambda$. Recall $Y_t$ is obtained by gluing rectangles of dimensions given by  the branches of the train-track $\mathcal{T}_\epsilon$ on $X_t$, which has euclidean width proportional to $t$ by Lemma \ref{lem:widerect}. Hence by construction, as $t\to \infty$, vertical heights on the singular flat surfaces $Y_t$ remain the same, but the horizontal widths are proportional to $t$. This is precisely the definition of surfaces along a Teichm\"{u}ller ray (see \S2.4). 
\end{proof}

\bigskip
\textit{Remark.}  The $t$ above is not  the arclength parametrization. It is related to the latter by $t = e^s$ (here $s$ is the distance along the ray).

\subsubsection*{Choice of $\epsilon^\prime$} The $\epsilon^\prime>0$ for the  train-track neighborhood $\mathcal{T}_{\epsilon^\prime}$ in the beginning of the section is chosen to be sufficiently small such that all branches (except the annular ones corresponding to any closed curve component of $\lambda$) have height greater than $H_1$ of Proposition \ref{prop:adg}. (See Lemma \ref{lem:thinrect}.) We shall also assume $\epsilon^\prime <\epsilon$ so that the discussion in \S3 applies.

\subsection{Step 4: Constructing the map $f$}
The goal is to construct an almost-conformal map $f:X_t\to Y_t$.

\subsubsection*{Mapping the pieces}
Recall the decomposition of the surfaces $X_t = S_t^1\sqcup \cdots \sqcup S^k_t$ and $Y_t = Y_t^1 \sqcup \cdots Y_t^k$ as in the previous section.

\begin{lem}
For all sufficiently large $t$, for each $1\leq j\leq k$, there is a $(1+ C\epsilon)$-quasiconformal embedding
\begin{equation}\label{eq:ej}
e^j: S_t^j \to S_\infty^j
\end{equation}
that is height-preserving on the vertical sides of the boundary and  $(\epsilon,3\epsilon)$-almost-isometric on the horizontal sides. Moreover, the image contains a truncation $S_{H_1}^j$ of the surface $S_\infty^j$.
\end{lem}
\begin{proof}
 $S_t^j$ consists of the hyperbolic surface $\widehat{S^j}$ that is a truncation of $S^j$ along horocyclic arcs, with adjacent rectangles through which leaves of the lamination $\lambda$ pass, and which thus have a grafted euclidean part. By our choice of $\epsilon^\prime$ (see the end of \S7.4) the heights of these rectangles are greater than $H_1$.\\
 
By Lemma \ref{lem:model}, for sufficiently large $t$ each of these rectangles admit almost-conformal maps to euclidean rectangles of identical heights and \textit{euclidean} widths, that are $(\epsilon,\epsilon)$-good on the boundary.  Together with an isometry on $\widehat{S^j}$ these maps can be pieced together to give the required embedding $e^j$ in $S_\infty^j$. The image of this map comprises euclidean rectangles and annuli appended to each geodesic side of $\widehat{S^j}$ .\\

By Lemma \ref{lem:widerect} for sufficiently large $t$ the euclidean widths of the appended rectangles or annuli are also all greater than $H_1$. Hence the embedded image in $S_\infty^j$ contains the truncation $S_{H_1}^j$.\\

A horizontal side of $S_t^j$ consists of two horizontal sides of adjacent rectangles, separated by a short horocyclic arc. The embedding is $(\epsilon,\epsilon)$-almost-isometric on the sides of the rectangle, and isometric on the horocyclic arc, and hence the concatenation is $(\epsilon, 3\epsilon)$-almost-isometric.
\end{proof}

\begin{prop}\label{prop:step3} For sufficiently large $t$, and for each $1\leq j\leq k$,  there is a $(1+C\epsilon)$-quasiconformal map
\begin{equation*}
F^j:S_t^j\to Y_t^j
\end{equation*}
that is $(\epsilon,4\epsilon)$-good on the boundary.
\end{prop}
\begin{proof}
By the previous lemma, for sufficiently large $t$,  we have an almost-conformal embedding $e^j:S_t^j \to S_\infty^j$ such that the image contains the truncation $S_{H_1}^j$. Postcomposing with $h^j$ of Proposition \ref{prop:adg} (restricted to this image), we get a $(1+C\epsilon)$-quasiconformal embedding $F^j: S_t^j\to Y_\infty^j$. By the fact that $e^j$ is height-preserving, and the remark following Proposition \ref{prop:adg}, the image is the truncation $Y_t^j$ of $Y_\infty^j$ (see Definition \ref{defn:trunc}). Moreover, since the additive errors of almost-isometries add up under composition, the composition yields an  $(\epsilon,4\epsilon)$-almost-isometry on the horizontal sides.
\end{proof}

\subsubsection*{The final map of the grafted surface}

\begin{prop}\label{prop:min} For sufficiently large $t$, there is a $(1+C\epsilon)$-quasiconformal homeomorphism $f:X_t \to Y_t$, homotopic to the identity map.
\end{prop}
\begin{proof}
By  Proposition \ref{prop:step3} for sufficiently large $t$ we have an almost-conformal map 
\begin{equation*}
F^j:S_t^j\to Y_t^j
\end{equation*}
for each $1\leq j\leq k$ that is $(\epsilon,4\epsilon)$-good on the boundary.\\

The surface $X_t$ comprises of the $S_t^j$-s by (\ref{eq:xtdec}) and the surface  $Y_t$ is obtained by gluing the singular flat pieces $Y_t^j$ (see Definition \ref{defn:yt}). However the maps above differ from the gluing maps on the boundary.  For the pieces to fit together continuously along the boundary, one has to post-compose with a ``correcting map" $F^j_\partial : \partial Y_t^j\to \partial Y_t^j$. Since all the boundary-maps constructed so far are $(\epsilon,M\epsilon)$-good (for $M>0$ a universal constant), so is $F^j_\partial$.  Recall that $Y_t^j$ comprises a collection of rectangles - by  Lemma \ref{lem:ext},  for sufficiently large $t$ one can extend $F^j_\partial$ to an almost-conformal self-map of $Y_t^j$.  One can now adjust the maps $F^j$ by post-composing with these almost-conformal maps. The composition then restricts to the desired map on the boundary, and these adjusted almost-conformal maps glue up to define an almost-conformal map  $f:X_t\to Y_t$. Since the conformal homeomorphism $g$ in (\ref{eq:mapg}) is homotopic to the identity map, and the gluings of the truncations preserve the marking, the final map is homotopic to the identity map as claimed.\end{proof}

\bigskip
This completes the proof of Theorem \ref{thm:thm1} (see the discussion in \S6.2).\\

By the remark on parametrization following Lemma \ref{lem:Yt}, if $Teich_{t\lambda}Y$ denotes the surface along the Teichm\"{u}ller ray from $Y$ in the direction determined by  $\lambda$, we in fact have:
\begin{equation}
d_\mathcal{T}(gr_{e^2t\lambda} X, Teich_{t\lambda} Y)\to 0 .
\end{equation}
 
 \bigskip
 \appendix
 
 \section{Generalized half-plane surfaces}
 
In this section we briefly recall the work in \cite{Gup25} and note a generalization (Theorem \ref{thm:hpd}) that is used the proof of Theorem \ref{thm:thm1} (see \S7.2).\\

\textit{Notation.} As a minor change of convention from \cite{Gup25},  in this paper we have switched what we call the ``horizontal" and ``vertical" directions for the quadratic differential metric (see \S2.2 for definitions). In particular, the euclidean ``half-planes"  below should be thought of as those bounded by a \textit{vertical} line on the plane, and the foliation by straight lines parallel to the boundary is its \textit{vertical} foliation.
 
 \subsection{Definitions}
 
 \begin{defn}[Half-plane surface] 
Let $\{H_i\}_{1\leq i\leq N}$ be a collection of $N\geq 2$ euclidean half planes and let $\mathcal{I}$ be a finite partition into sub-intervals of the boundaries of these half-planes. A \textit{half-plane surface} $\Sigma$ is a complete singular flat surface obtained by gluings by isometries amongst intervals from $\mathcal{I}$. 
\end{defn}

\textit{Remark.} The boundaries of the half-planes form a \textit{metric spine} of the resulting surface, so alternatively a half-plane surface can be thought of as a gluing of half-planes to an infinite-length metric graph to form a complete singular flat surface.
 
 \begin{defn}  A half-plane surface as above is equipped with a meromorphic quadratic differential $q$ called the \textit{half-plane differential}  that restricts to $dz^2$ in the usual coordinates on each half-plane.
  \end{defn}
  
 \begin{defn}[Local data at poles]\label{defn:locdat} The poles of a half-plane differential $q$ are at the ``punctures at infinity" of the half-plane surface. The \textit{residue} at a pole $p$ is the absolute value of the integral
\begin{equation*}
c = \int\limits_{\gamma} \sqrt{q}
\end{equation*}
where $\gamma$ is a simple closed curve enclosing $p$ and contained in a chart where one can define $\pm\sqrt{q}$.\\
If in local coordinates $q$ has the expansion:
 \begin{equation*}
 q = \left(\frac{a_n}{z^n} + \frac{a_{n-1}}{z^{n-1}} +\cdots \frac{a_{1}}{z} + a_0 +  \cdots\right)dz^2
 \end{equation*}
then $n$ is the \textit{order} and $a_n$ is the \textit{leading order term}.
 \end{defn}
 
 \textit{Remarks.} 1. A neighborhood of the poles are isometric to planar ends. The order equals $h+2$, where $h$ is the number of half-planes in the end, and the residue of the half-plane differential equals the metric residue as in Definition \ref{defn:pend}. (See Thm 2.6 in \cite{Gup25}.)\\
 
 2. It follows from definitions that the \textit{local data} of $q$ at the pole also satisfy the properties:
\begin{center}
($\ast$) \hspace{.3in} \textit{the residue is zero if the order is even}\\
($\ast\ast$) \hspace{.3in} \textit{each order is greater than or equal to $4$.}
\end{center}

\medskip
We also recall the following result (Lemma 3.8 of \cite{Gup25}):

\begin{lem}\label{lem:lot}
Let $f:\mathbb{D}\to \mathbb{C}$ be a univalent conformal map such that $f(0)=0$ and let $q$ be a meromorphic quadratic differential on $\mathbb{C}$ having the local expression
\begin{equation*}
q(z)dz^2 = \left(\frac{a_n}{z^n} + \frac{a_{n-1}}{z^{n-1}} +\cdots \frac{a_{1}}{z} + a_0 +  \cdots\right)dz^2
\end{equation*}
in the usual $z$-coordinates. Then the pullback quadratic differential $f^\ast q$ on $\mathbb{D}$ has leading order term equal to $\left| f^\prime(0)\right|^{2-n}a_n$ at the pole at $0$.
\end{lem}
  
  \subsection{The existence result}
  
 \begin{thm}[\cite{Gup25}]\label{thm:hpd1}
Given a Riemann surface $\Sigma$ with a set of points $P$  and prescribed local data $\mathcal{D}$ of orders, residues and leading order terms satisfying ($\ast$) and ($\ast\ast$) , there exists a half-plane surface $\Sigma_{D}$ and a conformal homeomorphism $g:\Sigma\setminus P\to \Sigma_{D}$ such that the local data of the corresponding half-plane differential is $\mathcal{D}$.\\
(The only exception is for the Riemann sphere with one marked point with a pole of order $4$, in which case the residue must equal zero.)
\end{thm}

 This can be thought of as an existence result of a meromorphic quadratic differential with prescribed poles of order at least $4$, that have a global ``half-plane structure" structure as described above.\\
 
 Here, we shall include poles of order two, which corresponds to half-infinite cylinders. (See Thm 2.3 in \cite{Gup25}.)

 \begin{defn} \label{defn:ghp} A \textit{generalized} half-plane surface is a complete singular-flat surface of infinite area obtained by gluing half-planes and half-infinite cylinders along a metric graph that forms a spine of the resulting punctured Riemann surface.
\end{defn}

 \begin{defn}[Data at a double pole]\label{defn:locdp} A  pole of order two of a generalized half-plane surface has a local expression of the form $\frac{C^2}{z^2}dz^2$ and has as an associated positive real number $C$, that we call its \textit{residue}. This is independent of the choice of local coordinates, and equals ($1/2\pi$) times the circumference of the corresponding half-infinite cylinder.
 \end{defn}
 
 We restate the theorem mentioned in \S1 slightly more precisely:

\begin{thm}\label{thm:hpd} Let $\Sigma$ be a Riemann surface with a set $P$ of $n$ marked points let $\mathcal{D}$ be local data satisfying ($\ast$) and ($\ast\ast$) for orders not equal to two. Then there is a corresponding  generalized half-plane surface $\Sigma_\mathcal{D}$ and a conformal homeomorphism
 \begin{equation*}
g:\Sigma\setminus P \to \Sigma_\mathcal{D}
\end{equation*}
that is homotopic to the identity map.\\
(The only exceptions are for the Riemann sphere with exactly one pole of order $4$, or exactly two poles of order $2$.)
\end{thm}

The proof of Theorem \ref{thm:hpd1} easily generalizes to include these half-infinite cylinders and give a proof of this result. In the next section we provide a sketch of this proof, referring to \cite{Gup25} for details.

 \subsection{Sketch of the proof of Theorem \ref{thm:hpd}}
 
 A ``generalized" half-plane surface is allowed to have half-infinite cylinders, in addition to half-planes. We shall follow the outline of the proof in \cite{Gup25} (see \S5 and \S12 of that paper) with the additional discussion for  the poles of order two. For all details see \cite{Gup25}.\\

We already  have a choice of coordinate charts around the points of $P$ on $\Sigma$, since the leading order terms of the prescribed local data $\mathcal{D}$ depend on such a choice. Pick one such chart $U$ containing the pole $p$, and conformally identify the pair with $(\mathbb{D},0)$. Let the local data associated with this pole consist of the order $n$, residue $a$ and leading order term $c$.

\subsubsection*{Quadrupling}
Produce an exhaustion of $\Sigma \setminus P$ by compact subsurfaces $\Sigma_i$ by excising, from each coordinate disk as above, the subdisk $U_i$ of radius $2^{-i}$.\\

We shall put a singular flat metric on each $\Sigma_i$ that we can complete to form a (generalized) half-plane surface $\Sigma_i^\prime$:\\

For the boundary component $\partial U_i$, mark off ($n-2$) disjoint arcs on it (no arcs for order two).  Take two copies of the surface $\Sigma_i$, and ``double" across these boundary arcs, that is,  glue the arcs on corresponding boundary components by an anti-conformal involution. If some orders are greater than or equal to $4$, the doubled surface has ``slits" corresponding to the complementary arcs on each of the boundary components which are not glued. Now form a \textit{closed} Riemann surface $\widehat{\Sigma_i}$ by doubling across these slits.

\subsubsection*{Singular flat metrics}

On  $\widehat{\Sigma_i}$ we have disjoint homotopy classes of simple closed curves around each of the slits glued in the second doubling step ($n$ of them associated with the pole $p$).  By a theorem of Jenkins and Strebel (see \cite{Streb}) there is a holomorphic quadratic differential which in the induced singular flat metric comprises metric cylinders glued together along their boundaries.  Quotienting back by the involutions of the two doubling steps, one gets a singular flat metric on the surface $\Sigma_i$ we started with. Each metric cylinder gives a rectangle as a quotient, for higher-order poles, and an annulus for a double-order pole.   Each boundary component  of the singular flat surface $\Sigma_i$ is  either closed (for the latter case) or ``polygonal" with alternating vertical and horizontal sides, corresponding to the boundary arcs we chose and their complementary arcs. \\

In our application of the Jenkins-Strebel theorem one can also prescribe the ``heights" of the cylinders, which in the quotient gives the lengths of the ``vertical" sides. We prescribe these  such that their alternating sum equals the residue for higher order poles. For a double-order pole, we choose the height (circumference of the annulus) equal to $2\pi$ times its residue (Definition \ref{defn:locdp}).\\

 By choosing the arcs in the initial step appropriately one can prescribe the circumferences of the cylinders. (See \S7 of \cite{Gup1}.) Together with the prescribed ``heights" this gives complete control on the dimensions of these polygonal boundaries. In particular, we choose these dimensions such that the boundary component corresponding to $p$ is isometric to the boundary of a truncation of a planar end (of residue $a$) at height $H_i$, where
\begin{equation}\label{eq:hi}
H_i=\left( H_0\cdot 2^i\right)^{n/2}
\end{equation}
(As we shall see later, $H_0$ is a constant that we choose for each pole.)\\
 
For each, we glue in a appropriate planar end (see Definition \ref{defn:pend}) or half-infinite cylinder, to get a generalized half-plane surface $\Sigma_i^\prime$.

 \subsubsection*{A geometric limit}
 
 Since $H_i\to \infty$ as $i\to \infty$ in (\ref{eq:hi}), the planar end one glues in to form $\Sigma_i^\prime$ gets smaller as a conformal disk, and it is easy to check from the definition in \S5.1 that  the sequence of  singular-flat surfaces $\{\Sigma_i\}_{i\geq 1 }$ have $\Sigma\setminus P$ as a conformal limit.\\
 
 On the other hand, one can show that the sequence of $\{\Sigma_i\}_{i\geq 1}$ has a generalized half-place surface  $\Sigma_{\mathcal{D}}$ as a conformal limit. As in \cite{Gup25}, the proof of this breaks into two steps: \\
 
Firstly, one can show that the corresponding  sequence of meromorphic quadratic differentials $q_i$  converges to one with the same local data $\mathcal{D}$, in the meromorphic quadratic differential bundle over $\mathcal{T}_g$ corresponding to the associated divisor. This holds because of the geometric control in (\ref{eq:hi}) - this makes the planar end $\mathcal{P}_{H_i}$ glued in  conformally comparable to the subdisk of radius $2^{-i}$ excised from $U$ (see Lemma 7.4 of \cite{Gup25}). Namely, there is an almost-conformal map between the pairs $(\mathbb{D}, B_{1/2^i})$ and $(U, \mathcal{P}_{H_i})$. On the (generalized) half-plane surface $\Sigma^\prime_i$, the disk $U$ can be identified with a planar domain via a conformal embedding $f_i$. Moreover  the subdisk $U_i$ is the preimage by $f$ of the planar end $\mathcal{P}_{H_i}$. The meromorphic quadratic differential $q_i$ is then a pullback of some fixed differential $q_0$ on $\mathbb{C}$ by $f_i$ (see the remark following Definition \ref{defn:pend}). By the almost-conformal correspondence above, there is a control on diameters of the planar domains involved, which gives a uniform bounds of the derivative of $f_i$ at $p$. The sequence $f_i$ then forms a normal family,  and $f_i\to f$ and $q_i\to q$. On $U$, $q$ restricts to the pullback of $q_0$ by $f$.\\
 
 In the second step, one shows that $q$ is also a generalized half-plane differential, that is, the sequence $\Sigma_i^\prime$ in fact converges to a generalized half-plane surface.  This latter limit $\Sigma_\mathcal{D}$ is built by attaching  half-planes and half-infinite cylinders along a metric spine that the metric spines of $\Sigma_i^\prime$ converge to, after passing to a subsequence. One needs an argument to show that this limiting spine has the same topology.  It suffices to show that one can choose a subsequence such that the metric graphs are identical as \textit{marked graphs}: Any edge of a metric spine along the sequence of converging $q_i$ has an adjacent collar of area proportional to its length, and along the sequence they cannot accumulate an increasing amount of ``twists"  about any non-trivial simple  closed curve since that contributes an increasing $q$-area to an embedded annulus about the curve. A sequence of spines having bounded twists about any simple closed curve is topologically identical after passing to a subsequence. Any cycle in the metric graph will then have a lower bound on its $q_i$-length (since the $q_i$-s are converging). Hence no cycles collapse, and $\Sigma_\mathcal{D}$ has the same topology. Almost conformal maps can then be built to the limiting surface by ``diffusing out" any collapse of a sub-graph that is tree (or forest), to the adjacent half-planes.\\

The sequence of these singular-flat surfaces $\{\Sigma_i\}_{i\geq 1}$ then has conformal limit both $\Sigma \setminus P$ and $\Sigma_{\mathcal{D}}$.  By Lemma \ref{lem:unique} there is a conformal homeomorphism $g:\Sigma\setminus P \to \Sigma_\mathcal{D}$ as desired.

\subsubsection*{The leading order term}

Finally one needs to show that the higher-order poles of the limiting meromorphic quadratic differential $q$ above have the required leading order terms. This is done by adjusting the constant term $H_0$ in (\ref{eq:hi}). The diameters of the associated planar domains $\mathcal{P}_{H_i}$ can take any value between  $0$ and $\infty$ by adjusting appropriately, and so do the derivatives at $p$ of the sequence of conformal embeddings $f_i$ (see above), and its limit $f$. By Lemma \ref{lem:lot} the leading order term is determined by $q_0$ and this derivative, and can be prescribed arbitrarily.

\bibliographystyle{amsalpha}
\bibliography{qmref}

\end{document}